\documentclass[10pt, a4paper]{amsart}
\usepackage[dvips,final]{graphics}
\usepackage{array}
\usepackage{arydshln}
\usepackage[makeroom]{cancel}
 \usepackage[all]{xy}
 \usepackage{url}
\usepackage{multirow, blkarray}
\usepackage{booktabs}
\usepackage{textcomp}
\usepackage[final]{epsfig}
\usepackage{color}
\usepackage[T1]{fontenc}      
\usepackage[english,french]{babel}
\usepackage[utf8]{inputenc}
\usepackage{blindtext}

\usepackage{amsfonts,amscd,array, mathdots, epigraph}
\usepackage{amsmath}
\usepackage{amssymb}
\usepackage{amsthm}
\usepackage{mathrsfs}
\usepackage{stmaryrd}
\usepackage{diagbox}
\usepackage{enumitem}
\usepackage{ulem}

\usepackage{ulem}
\usepackage{tikz}
\usepackage{xcolor}
\usepackage{multicol}

\usepackage{fullpage}

\newtheorem{theorem}{Théorème}[section]

\newtheorem{proposition}[theorem]{Proposition}
\newtheorem{corollary}[theorem]{Corollaire}

\newtheorem*{examples}{Exemples}
\newtheorem*{example}{Exemple}
\newtheorem{lemma}[theorem]{Lemme}
\newtheorem{algor}[theorem]{Algorithme}

\theoremstyle{definition}
\newtheorem{definition}[theorem]{Définition}
\newtheorem*{remark}{Remarque}

\numberwithin{equation}{section}

\newcommand{\algo}[7]
{
\begin{algor}{{\tt #1}{\rm (#2)}}\label{#7}\end{algor}
\vspace{-6pt}\noindent{\it #3}.

{\bf Input:} #4

{\bf Output:} #5

\newcounter{#7}
\begin{list}{\textbf{\arabic{#7}.}}{\usecounter{#7}}
#6\end{list}\vspace{3pt}
}

\newcommand{\bZ}{\mathbb{Z}}
\newcommand{\bC}{\mathbb{C}}
\newcommand{\bN}{\mathbb{N}}
\newcommand{\bM}{\mathbb{N^{*}}}

\DeclareMathOperator{\id}{Id}

\DeclareMathOperator{\SL}{SL}
\newcommand{\cat}{\cup}

\title[Comptage des quiddités sur les corps finis et sur quelques anneaux $\mathbb{Z}/N\mathbb{Z}$]{Comptage des quiddités sur les corps finis et sur quelques anneaux $\mathbb{Z}/N\mathbb{Z}$}

\author{Michael Cuntz$^{a}$, Flavien Mabilat$^{b}$}

\date{}


\address{$^{a}$ Leibniz Universität Hannover\\
Institut für Algebra, Zahlentheorie und Diskrete Mathematik, Fakultät für Mathematik und Physik, Welfengarten 1, D-
30167 Hannover, Germany}
\email{cuntz@math.uni-hannover.de}
\address{$^{b}$ Laboratoire de Mathématiques de Reims,
UMR9008 CNRS et Université de Reims Champagne-Ardenne, 
U.F.R. Sciences Exactes et Naturelles 
Moulin de la Housse - BP 1039 
51687 Reims cedex 2,
France}
\email{flavien.mabilat@univ-reims.fr}

\begin{document}

\maketitle

\selectlanguage{french}
\begin{abstract}
Les $\lambda$-quiddités de taille $n$ sont des $n$-uplets d'éléments d'un ensemble fixé, solutions d'une équation matricielle apparaissant lors de l'étude des frises de Coxeter. Celles-ci peuvent être considérées sur divers ensembles avec des structures très variables d'un ensemble à l'autre. L'objectif principal de ce texte est d'obtenir des formules explicites donnant le nombre de $\lambda$-quiddités de taille $n$ sur un corps fini et sur les anneaux $\mathbb{Z}/N\mathbb{Z}$ lorsque $N$ est de la forme $4m$ avec $m$ sans facteur carré. On donnera également quelques éléments sur le comportement asymptotique du nombre de $\lambda$-quiddités vérifiant une condition d'irréductibilité sur $\mathbb{Z}/N\mathbb{Z}$ lorsque $N$ tend vers l'infini.
\end{abstract}

\selectlanguage{english}
\begin{abstract}
The $\lambda$-quiddities of size $n$ are $n$-tuples of elements of a fixed set, solutions of a matrix equation appearing in the study of Coxeter's friezes. These can be considered on various sets with very different structures from one set to another. The main objective of this text is to obtain explicit formulas giving the number of $\lambda$-quiddities of size $n$ over finite fields and over the rings $\mathbb{Z}/N\mathbb{Z}$ with $N=4m$ and $m$ square free. We will also give some elements about the asymptotic behavior of the number of $\lambda$-quiddities verifying an irreducibility condition over $\mathbb{Z}/N\mathbb{Z}$ when $N$ goes to the infinity.
\end{abstract}

\selectlanguage{french}

\ \\
\begin{flushright}
 \textit{\og Why should only Fermat have a Little Theorem? \fg} 
\\ John H.\ Conway, \textit{A Characterization of the Equilateral Triangles and Some Consequences}
\end{flushright}
\ \\

\section{Introduction}
\label{Intro}

Les frises de Coxeter sont des arrangements de nombres dans le plan vérifiant certaines relations arithmétiques (voir section \ref{cox}). Introduites au début des années soixante-dix par H.\ S.\ M.\ Coxeter (voir \cite{Cox}), elles sont aujourd'hui au centre de très nombreux travaux, du fait notamment de leurs liens étendus avec de multiples branches des mathématiques (voir par exemple \cite{Mo1}). En particulier, l'étude de ces objets passe par la considération de l'équation matricielle suivante:
\[M_{n}(a_1,\ldots,a_n):=\begin{pmatrix}
   a_{n} & -1 \\[4pt]
    1    & 0 
   \end{pmatrix}
\begin{pmatrix}
   a_{n-1} & -1 \\[4pt]
    1    & 0 
   \end{pmatrix}
   \cdots
   \begin{pmatrix}
   a_{1} & -1 \\[4pt]
    1    & 0 
    \end{pmatrix}=-\id.\]
\noindent En effet, ce sont les solutions de cette équation qui permettent de construire des frises de Coxeter (voir \cite{BR} et \cite{CH} proposition 2.4 ainsi que la section \ref{cox}). Cela conduit notamment à chercher l'ensemble des solutions de cette équation ainsi que le nombre de solutions de taille fixée. Dans cette optique, S.\ Morier-Genoud a calculé ce dernier dans le cas où les $a_{i}$ appartiennent à un corps fini (voir \cite{Mo2} Théorème 1 et la section \ref{cox}). 

Par ailleurs, l'étude de l'équation précédente amène naturellement à considérer l'équation généralisée ci-dessous :
\[M_{n}(a_1,\ldots,a_n)=\pm \id.\]

\noindent Les solutions de celles-ci sont appelées $\lambda$-quiddités (voir \cite{C}) et l'objectif principal est de mieux connaître ces dernières pour un certain nombre d'ensembles. Outre la recherche de résultats généraux, tel le nombre de $\lambda$-quiddités de taille fixée, on cherche également à restreindre l'étude des solutions à la recherche de $\lambda$-quiddités particulières. Pour effectuer cela, on introduit une notion de solutions irréductibles, basée sur une opération entre uplets d'éléments d'un ensemble fixé (voir \cite{C} et la section suivante). Cela permet de réduire l'étude des $\lambda$-quiddités à la bonne connaissance des solutions irréductibles. Dans cette optique, on dispose d'une classification précise des $\lambda$-quiddités irréductibles sur plusieurs sous-ensembles de $\bC$ (voir notamment \cite{C}), ainsi que d'un certain nombre d'informations sur celles-ci pour les cas des anneaux $\bZ/N\bZ$ (voir par exemple \cite{M1, M5}). De plus, on a également une construction récursive des $\lambda$-quiddités sur $\mathbb{N}^{*}$ et des formules permettant de compter ces dernières (voir \cite{CO,O}).

On s'intéresse ici à la recherche des $\lambda$-quiddités sur les corps finis et sur les anneaux $\bZ/N\bZ$. En d'autres termes, on considère la résolution sur un anneau $A$ fixé de l'équation.
\begin{equation}
\label{p}
\tag{$E_{A}$}
M_{n}(a_1,\ldots,a_n)=\pm \id.
\end{equation}
L'étude de cette équation est d'autant plus intéressante qu'elle est liée à la recherche des différentes écritures des éléments des sous-groupes de congruence ci-dessous :
\[\hat{\Gamma}(N)=\{C \in \SL_{2}(\mathbb{Z})~{\rm tel~que}~C= \pm \id~( {\rm mod}~N)\}.\]
En effet, on sait que toutes les matrices de $\SL_{2}(\bZ)$ peuvent s'écrire sous la forme $M_{n}(a_1,\ldots,a_n)$, avec les $a_{i}$ des entiers naturels strictement positifs. Cette écriture n'étant pas unique, on est naturellement amené à chercher toutes les écritures de cette forme pour une matrice, ou un ensemble de matrices, donné.

Notre objectif dans cet article est d'obtenir le nombre de $\lambda$-quiddités sur les corps finis et sur les anneaux $\bZ/N\bZ$ lorsque $N=4m$ avec $m$ sans facteur carré. Plus précisément, on va démontrer le théorème ci-dessous qui complète le résultat déjà obtenu par S.\ Morier-Genoud. Avant de donner des énoncés précis, on introduit quelques notations. Si $X$ est un ensemble fini, $\left|X\right|$ représente le cardinal de $X$.
Pour $q$ la puissance d'un nombre premier $p$, $B \in \SL_{2}(\mathbb{F}_{q})$ et $n \in \mathbb{N}^{*}$, on note
\[ u_{n,q}^{B} := |\{ (a_1,\ldots,a_n)\in\mathbb{F}_{q}^n \mid M_{n}(a_{1},\ldots,a_{n})=B\}|. \]

\noindent Notons $u_{n,q}^{\pm}:=u_{n,q}^{\pm\id}$.
De plus, si $m \in \bM$ et $q \in \mathbb{N}_{>1}$, on note :
\[ [m]_{q}:=\frac{q^{m}-1}{q-1}, \quad \binom{m}{2}_{q}:=\frac{(q^{m}-1)(q^{m-1}-1)}{(q-1)(q^{2}-1)}. \]

\noindent Par ailleurs, on aura besoin à plusieurs reprises de la formule suivante : $\frac{q^{2m-2}-1}{q^{2}-1}+q\binom{m-1}{2}_{q}=\binom{m}{2}_{q}$.

\begin{theorem}
\label{25}
Soient $q$ la puissance d'un nombre premier $p>2$ et $n\in\mathbb{N}$, $n>4$.

\smallskip
i) Si $n$ est impair alors on a $u_{n, q}^{+}=u_{n, q}^{-}=\left[\frac{n-1}{2}\right]_{q^{2}}$.

\smallskip
ii) Si $n$ est pair alors il existe $m \in \bM$ tel que $n=2m$. 
\begin{itemize}

\item Si $m$ est pair on a : $u_{n,q}^{+}=(q-1)\binom{m}{2}_{q}+q^{m-1}$.
\item Si $m \geq 3$ est impair on a : $u_{n,q}^{+}=(q-1)\binom{m}{2}_{q}$.

\end{itemize}

\end{theorem}

Couplé au théorème \ref{36}, on obtient l'ensemble des formules de comptage des $\lambda$-quiddités sur les corps finis. Ensuite, on établira une formule générale de récurrence concernant le nombre de solutions de l'équation $M_{n}(a_1,\ldots,a_n)=\pm B$, avec $B$ un élément quelconque de $\SL_{2}(\mathbb{F}_{q})$ , qui permettra de retrouver de façon indépendante l'ensemble des formules souhaitées. 

\begin{theorem}
\label{thm412}
Soient $q$ la puissance d'un nombre premier $p$, $B \in \SL_{2}(\mathbb{F}_{q})$ et $n \in \mathbb{N}$, $n>4$. On a la relation suivante:
\[ u_{n,q}^{B} = (q-1)(u_{n-1,q}^{B}-q u_{n-3,q}^{-B}) + q u_{n-2,q}^{-B} + q(u_{n-2,q}^{B}-q u_{n-4,q}^{-B}). \]
\end{theorem}

\indent On démontrera également, dans la section \ref{preuve26}, des formules pour le nombre de $\lambda$-quiddités sur $\bZ/4\bZ$. Pour cela, on note :
\[ w_{n,N}^{\pm} := |\{ (a_1,\ldots,a_n)\in(\mathbb{Z}/N\mathbb{Z})^n \mid M_{n}(a_{1},\ldots,a_{n})=\pm \id\}|. \]

\begin{theorem}
\label{26}

Soit $n$ un entier naturel supérieur à 3.

\smallskip
i) Si $n$ est impair alors on a l'égalité suivante : 
\[w_{n, 4}^{+}=w_{n, 4}^{-}=\frac{4^{n-2}-2^{n-3}}{3}.\]

\smallskip
ii) Si $n$ est pair alors il existe $m \in \bM$ tel que $n=2m$.
\begin{itemize}

\item Si $m$ est pair on a :
\[w_{n,4}^{+}=\frac{4^{n-2}+2^{n-1}}{3}~~~et~~~w_{n,4}^{-}=\frac{4^{n-2}- 2^{n-2}}{3}.\]
\item Si $m$ est impair on a :
\[w_{n,4}^{+}=\frac{4^{n-2}- 2^{n-2}}{3}~~~et~~~w_{n,4}^{-}=\frac{4^{n-2}+2^{n-1}}{3}.\]

\end{itemize}

\end{theorem}

Le lemme chinois, le théorème \ref{25}, le théorème \ref{36} et le résultat ci-dessus permettent d'obtenir des formules pour $w^{\pm}_{n,N}$ dans le cas où $N=4m$ avec $m$ sans facteur carré. Plus précisément on a :

\begin{corollary}
\label{261}

Soit $N=p_{1}\ldots p_{r}$ avec les $p_{i}$ des nombres premiers impairs deux à deux distincts. Soit $n \geq 2$, on a :
\[w_{n,N}^{\pm}=u_{n,p_{1}}^{\pm}u_{n,p_{2}}^{\pm}\ldots u_{n,p_{r}}^{\pm},~~~~~~w_{n,4N}^{\pm}=w_{n,4}^{\pm}u_{n,p_{1}}^{\pm}u_{n,p_{2}}^{\pm}\ldots u_{n,p_{r}}^{\pm}.\]

\end{corollary}

\noindent Notons qu'une formule générale de récurrence pour $w^{\pm}_{n,N}$ dépasse le cadre de cet article.

On s'intéressera pour finir aux solutions irréductibles (voir définition \ref{24}) de \eqref{p}, en essayant d'obtenir quelques  informations sur la suite $(v_{N})$, le nombre total de classes d'équivalence de solutions irréductibles de \eqref{p} sur $\bZ/N\bZ$:

\begin{theorem}
\label{27}
i) $\left(\frac{v_{N}}{N\sqrt{{\rm ln}(N)}}\right)$ n'est pas bornée.
\\ii) $\forall l \in \mathbb{R}^{+}$, $(v_{N}-lN)$ n'est pas bornée.

\end{theorem}

Ce résultat est prouvé dans la section \ref{preuve27}. Dans cette dernière, on donnera également le nombre de solutions irréductibles de \eqref{p} sur $\bZ/N\bZ$ pour $N \leq 16$, obtenu grâce à un programme informatique, fourni dans l'annexe \ref{A}.

\section{Quelques éléments sur les frises}
\label{elements_frises}

\subsection{Quiddités}
\label{quid}
Le but poursuivi par cette section est d'énoncer formellement un certain nombre de définitions évoquées dans l'introduction.
À moins que des précisions supplémentaires ne soient apportées, $A$ est un anneau commutatif unitaire, $q$ désigne la puissance d'un nombre premier $p$, $\mathbb{F}_{q}$ est le corps fini à $q$ éléments. $N$ désigne un entier naturel supérieur à 2 et, si $a \in \bZ$, on note $\overline{a}:=a+N\bZ$. On commence par donner une définition précise du concept de $\lambda$-quiddité.

\begin{definition}[\cite{C}, définition 2.2]
\label{21}

Soit $n \in \bM$. On dit que le $n$-uplet $(a_{1},\ldots,a_{n})$ d'éléments de $A$ est une $\lambda$-quiddité sur $A$ de taille $n$ si $(a_{1},\ldots,a_{n})$ est une solution de \eqref{p}, c'est-à-dire si $M_{n}(a_{1},\ldots,a_{n})=\pm \id.$ En cas d'absence d'ambiguïté, on parlera simplement de $\lambda$-quiddité.

\end{definition}

\noindent Pour procéder à l'étude des $\lambda$-quiddités, on utilisera les deux définitions ci-dessous :

\begin{definition}[\cite{C}, lemme 2.7]
\label{22}

Soient $(n,m) \in (\bM)^{2}$, $(a_{1},\ldots,a_{n}) \in A^{n}$ et $(b_{1},\ldots,b_{m}) \in A^{m}$. On définit l'opération suivante: \[(a_{1},\ldots,a_{n}) \oplus (b_{1},\ldots,b_{m}):= (a_{1}+b_{m},a_{2},\ldots,a_{n-1},a_{n}+b_{1},b_{2},\ldots,b_{m-1}).\] Le $(n+m-2)$-uplet ainsi obtenu est appelé la somme de $(a_{1},\ldots,a_{n})$ avec $(b_{1},\ldots,b_{m})$.

\end{definition}

\begin{examples} 

{\rm On suppose $N \geq 8$ et on se place sur $\bZ/N\bZ$.} 
\begin{itemize}
\item $(\overline{1},\overline{2},\overline{3}) \oplus (\overline{3},\overline{2},\overline{1}) = (\overline{2},\overline{2},\overline{6},\overline{2})$;
\item $(\overline{1},\overline{0},\overline{4}) \oplus (\overline{3},\overline{1},\overline{0},\overline{1}) = (\overline{2},\overline{0},\overline{7},\overline{1},\overline{0})$;
\item $(\overline{5},\overline{2},\overline{1},\overline{1},\overline{3}) \oplus (\overline{2},\overline{2},\overline{3},\overline{0}) = (\overline{5},\overline{2},\overline{1},\overline{1},\overline{5},\overline{2},\overline{3})$;
\item $n \geq 2$, $(\overline{a_{1}},\ldots,\overline{a_{n}}) \oplus (\overline{0},\overline{0}) = (\overline{0},\overline{0}) \oplus (\overline{a_{1}},\ldots,\overline{a_{n}})=(\overline{a_{1}},\ldots,\overline{a_{n}})$.
\end{itemize}

\end{examples}

L'opération $\oplus$ est très intéressante pour mener à bien l'étude des solutions de l'équation \eqref{p}. En effet, celle-ci possède la propriété suivante : si $(b_{1},\ldots,b_{m})$ est une $\lambda$-quiddité alors la somme $(a_{1},\ldots,a_{n}) \oplus (b_{1},\ldots,b_{m})$ est une $\lambda$-quiddité si et seulement si $(a_{1},\ldots,a_{n})$ est une $\lambda$-quiddité (voir \cite{C,WZ} et \cite{M1} proposition 3.7). Toutefois, il est important de noter que $\oplus$ n'est ni commutative ni associative (voir \cite{WZ} exemple 2.1).

\begin{definition}[\cite{C}, définition 2.5]
\label{23}

Soient $(a_{1},\ldots,a_{n})$ et $(b_{1},\ldots,b_{n})$ deux $n$-uplets d'éléments de $A$. On note $(a_{1},\ldots,a_{n}) \sim (b_{1},\ldots,b_{n})$ si $(b_{1},\ldots,b_{n})$ est obtenu par permutations circulaires de $(a_{1},\ldots,a_{n})$ ou de $(a_{n},\ldots,a_{1})$.

\end{definition}

On constate aisément que $\sim$ est une relation d'équivalence sur l'ensemble des $n$-uplets d'éléments de $A$. De plus, si un $n$-uplet d'éléments de $A$ est une $\lambda$-quiddité alors tout $n$-uplet d'éléments de $A$ qui lui est équivalent est aussi une $\lambda$-quiddité (voir \cite{C} proposition 2.6). Muni de ces deux définitions, on peut formuler la notion d'irréductibilité évoquée dans l'introduction.

\begin{definition}[\cite{C}, définition 2.9]
\label{24}

Une solution $(c_{1},\ldots,c_{n})$ avec $n \geq 3$ de \eqref{p} est dite réductible s'il existe une solution de \eqref{p} $(b_{1},\ldots,b_{l})$ et un $m$-uplet $(a_{1},\ldots,a_{m})$ d'éléments de $A$ tels que \begin{itemize}
\item $(c_{1},\ldots,c_{n}) \sim (a_{1},\ldots,a_{m}) \oplus (b_{1},\ldots,b_{l})$;
\item $m \geq 3$ et $l \geq 3$.
\end{itemize}
Une solution est dite irréductible si elle n'est pas réductible.

\end{definition}

\begin{remark} 

{\rm $(0,0)$ est toujours une $\lambda$-quiddité. En revanche, celle-ci n'est jamais considérée comme étant irréductible.}

\end{remark}

L'étude des solutions de \eqref{p} de petite taille permet de connaître exactement le nombre de $\lambda$-quiddités dans ces cas. On dispose notamment des résultats suivants (voir par exemple \cite{CH} exemple 2.7 ou \cite{M1} section 3.1 pour le détail des calculs) :

\begin{proposition}
i) \eqref{p} n'a pas de solution de taille 1.

\smallskip
ii) $(0,0)$ est l'unique solution de \eqref{p} de taille 2.

\smallskip
iii) $(1,1,1)$ et $(-1,-1,-1)$ sont les seules solutions de \eqref{p} de taille 3.

\smallskip
iv) Les solutions de \eqref{p} de taille 4 sont de la forme $(-a,b,a,-b)$ avec $ab=0$ et $(a,b,a,b)$ avec $ab=2$.

\end{proposition}

On en déduit qu'il existe zéro $\lambda$-quiddité de taille 1 et une $\lambda$-quiddité de taille 2. Si $1 \neq -1$ dans $A$ alors \eqref{p} a deux solutions de taille 3 (si $1=-1$ alors il y en a une seule). De plus, si $A$ est un corps fini de cardinal $q$ et de caractéristique différente de 2 alors \eqref{p} a $3q-2$ solutions de taille 4 (et $2q-1$ s'il est de caractéristique 2). On souhaite ici obtenir plus d'informations, c'est-à-dire compter le nombre de $\lambda$-quiddités de taille fixée lorsque $A$ est un corps fini.

\subsection{Frises de Coxeter}
\label{cox}

L'objectif de cette section est de donner les quelques éléments sur les frises de Coxeter nécessaires à la formulation du théorème de S. Morier-Genoud, évoqué dans la section précédente. Cela nous permettra également de préciser certains éléments de contexte fournis dans l'introduction. On commence par la définition suivante :

\begin{definition}[\cite{Cox}]
\label{31}

Une frise de Coxeter est un tableau de nombres vérifiant les conditions suivantes :
\begin{itemize}[label=$\circ$]
\item  le nombre de lignes est fini;
\item chaque ligne est infinie à gauche et à droite;
\item la première et la dernière ligne ne contiennent que des 1;
\item deux lignes consécutives sont disposées avec un décalage;
\item lorsque quatre éléments $a, b, c, d$ sont disposés de la façon suivante :
$$
\begin{array}{ccc}
&b&\\
a&&d\\
&c&
\end{array},
$$
la règle unimodulaire est vérifiée, c'est-à-dire $ad-bc=1$.
\end{itemize}

\noindent Si $m$ est le nombre de lignes de la frise alors $n=m-2$ est la largeur de la frise.

\end{definition}

\noindent On définit maintenant une classe particulière de frises de Coxeter.

\begin{definition}[\cite{BR}]
\label{32}

On peut étendre une frise de Coxeter en rajoutant en bas (resp. en haut) du tableau une ligne constituée entièrement de 0 suivie (resp. précédée) d'une ligne constituée uniquement de $-1$ (en prolongeant une frise de cette manière, la règle unimodulaire est toujours satisfaite). Prolongée ainsi, une frise de Coxeter est dite docile si pour tout éléments $a, b, \ldots, i$ présents dans la frise et disposés de la façon suivante :
\begin{center}
$
\begin{array}{ccccc}
&&c\\
&b&&f\\
a&&e&&i\\
&d&&h\\
&&g
\end{array}
$
\quad\quad \text{on a} \quad\quad
$\left|
\begin{array}{cccccc}
a & b & c\\
d & e & f\\
g & h & i\\
\end{array}
\right|=0.
$
\end{center}
\end{definition}

\indent En utilisant la règle d'unimodalité rappelée ci-dessus, on peut, à partir de la deuxième ligne (c'est-à-dire la ligne en-dessous de celle ne contenant que des 1), reconstruire toutes les lignes de la frise. On dispose également du résultat ci-dessous :

\begin{theorem}[\cite{Cox}]
\label{33}

Les lignes d'une frise de Coxeter docile de largeur $n$ sont périodiques de période $n+3$.

\end{theorem}

\begin{example}
{\rm Voici un exemple de frise de Coxeter de largeur 2 :}
$$
 \begin{array}{cccccccccccccccc}
\cdots&&1&& 1&&1&&1&&1
 \\[2pt]
&1&&3&&1&&2&&2&&\cdots
 \\[2pt]
\cdots&&2&&2&&1&&3&&1
 \\[2pt]
&1&& 1&&1&&1&&1&&\cdots
\end{array}
$$

\end{example}

Il existe un lien surprenant entre les frises de Coxeter à coefficients entiers strictement positifs et les triangulations de polygones convexes. Pour détailler celui-ci, on a d'abord besoin de la définition suivante, introduite dans \cite{CoCo} :

\begin{definition}

On considère une triangulation d'un polygone convexe à $n$ sommets $P_{1}P_{2} \ldots P_{n}$. On appelle quiddité associée à la triangulation la séquence $(a_{1},\ldots,a_{n})$ où $a_{i}$ est égal au nombre de triangles utilisant le sommet $P_{i}$.

\end{definition}

\noindent On dispose du résultat suivant:

\begin{theorem}[Conway-Coxeter,~\cite{CoCo}]
Soit $n \geq 3$.

\smallskip
i) La quiddité $(a_{1},\ldots,a_{n})$ associée à la triangulation d'un polygone convexe à $n$ sommets détermine la deuxième ligne d'une frise de largeur $n-3$. 

\smallskip
ii) Tout $n$-uplet d'entiers strictement positifs $(a_{1},\ldots,a_{n})$ qui définit la deuxième ligne d'une frise de largeur $n-3$ est la quiddité associée à la triangulation d'un polygone convexe à $n$ sommets.

\end{theorem}

Comme indiqué dans l'introduction, les frises de Coxeter sont reliées à l'étude de l'équation \eqref{p} par l'intermédiaire du résultat qui suit (voir \cite{BR} et \cite{CH} proposition 2.4).

\begin{proposition}
\label{34}

Soit $A$ un anneau commutatif unitaire. Le $n$-uplet $(a_{1},\ldots,a_{n})$ d'éléments de $A$ détermine la deuxième ligne d'une frise docile de largeur $n-3$ si et seulement si $M_{n}(a_{1},\ldots,a_{n})=-\id$. 

\end{proposition}

\begin{example}
{\rm Si on reprend la frise de l'exemple précédent, on a $M_{5}(1,3,1,2,2)=-\id$.
}
\end{example}

On souhaite maintenant compter le nombre de frises de Coxeter dociles dont les éléments appartiennent à un corps fini. Grâce à la proposition précédente, cela équivaut à calculer $u_{n,q}^{-}$.

\begin{theorem}[Morier-Genoud,~\cite{Mo2} Théorème 1]
\label{36}

Soient $q$ la puissance d'un nombre premier $p$ et $n\in\mathbb{N}$, $n>4$.

\smallskip
i) Si $n$ est impair alors $u_{n, q}^{-}=\left[\frac{n-1}{2}\right]_{q^{2}}$.

\smallskip
ii) Si $n$ est pair alors il existe $m \in \bM$ tel que $n=2m$.
\begin{itemize}

\item Si $p=2$, $u_{n,q}^{-}=(q-1)\binom{m}{2}_{q}+q^{m-1}$.
\item Si $p>2$ et $m$ pair on a : $u_{n,q}^{-}=(q-1)\binom{m}{2}_{q}$.
\item Si $p>2$ et $m \geq 3$ impair on a : $u_{n,q}^{-}=(q-1)\binom{m}{2}_{q}+q^{m-1}$.

\end{itemize}

\end{theorem}

\begin{remark}
{\rm À l'origine, le résultat donné dans \cite{Mo2} concerne uniquement le nombre de frises.
}
\end{remark}

Pour de nombreux autres éléments sur les frises et leurs applications, on peut consulter l'article \cite{Mo1} qui donne une présentation assez complète du sujet.

\subsection{Résultats préliminaires}
\label{chap41}

Le but de cette sous-section est de fournir un certain nombre de résultats utiles pour la suite et d'obtenir quelques éléments découlant immédiatement du théorème de S.\ Morier-Genoud. 

\subsubsection{Cas $p=2$}
On peut utiliser la proposition \ref{34} et le théorème \ref{36}
pour obtenir d'autres formules de dénombrement. On se propose ici d'appliquer ce cas aux triangulations de polygones via la définition ci-dessous :

\begin{definition}
\label{41}

Soit $P=P_{1}\ldots P_{n}$ un polygone convexe à $n$ sommets. On effectue une triangulation de $P$ et on note $(d_{1},\ldots,d_{n})$ la quiddité associée à celle-ci. La séquence de parité de la triangulation est le $n$-uplet $(\overline{d_{1}},\ldots,\overline{d_{n}})\in\mathbb{F}_2^n$, c'est-à-dire que pour tout $i$, $1\le i\le n$, $\overline{d_{i}}:=\overline{0}$ si un nombre pair de triangles utilise le sommet $P_{i}$ et $\overline{d_{i}}:=\overline{1}$ sinon.

\end{definition}

\begin{figure}[ht]
$
\shorthandoff{; :!?}
\xymatrix @!0 @R=0.40cm @C=0.5cm
{
&&1\ar@{-}[lldd] \ar@{-}[rr]&&5\ar@{-}[rrdd]\ar@{-}[dddddd]\ar@{-}[lldddddd]\ar@{-}[rrdddd]&
\\
&&&
\\
3\ar@{-}[dd]\ar@{-}[rrrruu]&&&&&& 1 \ar@{-}[dd]
\\
&&
\\
1&&&&&& 2
\\
&&&
\\
&&3 \ar@{-}[rr]\ar@{-}[lluuuu]\ar@{-}[lluu] &&2 \ar@{-}[rruu]
}
$ \quad \quad \quad $
\shorthandoff{; :!?}
\xymatrix @!0 @R=0.40cm @C=0.5cm
{
&&\overline{1}\ar@{-}[lldd] \ar@{-}[rr]&&\overline{1}\ar@{-}[rrdd]\ar@{-}[dddddd]\ar@{-}[lldddddd]\ar@{-}[rrdddd]&
\\
&&&
\\
\overline{1}\ar@{-}[dd]\ar@{-}[rrrruu]&&&&&& \overline{1} \ar@{-}[dd]
\\
&&
\\
\overline{1}&&&&&& \overline{0}
\\
&&&
\\
&&\overline{1} \ar@{-}[rr]\ar@{-}[lluuuu]\ar@{-}[lluu] &&\overline{0} \ar@{-}[rruu]
}
$
\caption{À gauche une triangulation et sa quiddité, à droite la séquence de parité de cette triangulation\label{figquidpond}}
\end{figure}
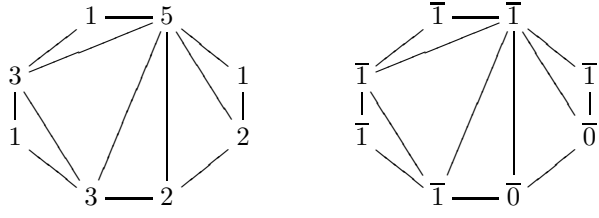
\begin{example}
\noindent {\rm La séquence de parité de la triangulation dans la figure \ref{figquidpond} est
$(\overline{1}, \overline{1}, \overline{1}, \overline{1}, \overline{0}, \overline{0}, \overline{1}, \overline{1})$. C'est une solution de $(E_{\mathbb{F}_2})=(E_{2})$.
}

\end{example}

\begin{proposition}[\cite{M0}, remarque 5.4]
\label{42}
Soit $n \geq 3$.

\smallskip
i) Toute séquence de parité d'une triangulation d'un polygone convexe à $n$ sommets est une solution de taille $n$ de $(E_{2})$ possédant un élément non nul.

\smallskip
ii) Si $(\overline{c_{1}},\ldots,\overline{c_{n}})$ est une solution de $(E_{2})$ et s'il existe un entier $i$ dans $[\![1;n]\!]$ tel que $\overline{c_{i}} \neq \overline{0}$ alors $(\overline{c_{1}},\ldots,\overline{c_{n}})$ est la séquence de parité d'une triangulation d'un polygone convexe à $n$ sommets.
\end{proposition}

Grâce à cela, on peut compter le nombre de séquences de parité d'une triangulation d'un polygone convexe.

\begin{proposition}
\label{43}
Soit $n \geq 3$ et $P$ un polygone convexe à $n$ sommets.

\smallskip
i) Si $n$ est impair, il y a $\left[\frac{n-1}{2}\right]_{4}$ séquences de parité possibles associées aux triangulations de $P$.

\smallskip
ii) Si $n=2m$ est pair, il y a $\binom{m}{2}_{2}+2^{m-1}-1$ séquences de parité possibles associées aux triangulations de $P$.
\end{proposition}

\begin{proof}

Par ce qui précède, le nombre de séquences de parité possibles associées aux triangulations de $P$ est égale au nombre de solutions non nulles de $(E_{2})$. Or, il n'y a pas de solution nulle de taille impaire. Donc, le théorème \ref{36} donne i). De plus, il y a exactement une solution nulle de taille paire égale à $2m$. Ainsi, en retirant 1 à la formule correspondante du théorème \ref{36}, on obtient ii).
\end{proof}

\subsubsection{Quelques opérations sur les $\lambda$-quiddités}

Le résultat suivant généralise la propriété 5.1 de \cite{CH}.

\begin{lemma}
\label{48}

Soient $n \geq 4$, $n=2m$ et $\lambda$ un élément inversible de $A$. Soit $(a_{1},\ldots,a_{n}) \in A^{n}$. Si $M_{n}(a_{1},\ldots,a_{n})=\epsilon \id$ avec $\epsilon \in \{1,-1\}$ alors $M_{n}(\lambda a_{1},\lambda^{-1} a_{2},\ldots,\lambda a_{2m-1},\lambda^{-1} a_{2m})=\epsilon \id$.

\end{lemma}
\begin{proof}
On vérifie par un calcul direct:
\[
\begin{pmatrix} \lambda^{-1} & 0 \\ 0 & 1 \end{pmatrix}
\begin{pmatrix} \lambda a & -1 \\ 1 & 0 \end{pmatrix}
\begin{pmatrix} \lambda^{-1} b & -1 \\ 1 & 0 \end{pmatrix}
\begin{pmatrix} \lambda & 0 \\ 0 & 1 \end{pmatrix}
=
\begin{pmatrix} a & -1 \\ 1 & 0 \end{pmatrix}
\begin{pmatrix} b & -1 \\ 1 & 0 \end{pmatrix}
\]
pour tout $(a,b)\in A^{2}$. Ainsi, \[\epsilon \id=M_{n}(a_{1},\ldots,a_{n})=\begin{pmatrix} \lambda & 0 \\ 0 & 1 \end{pmatrix}M_{n}(\lambda a_{1},\lambda^{-1} a_{2},\ldots,\lambda a_{2m-1},\lambda^{-1} a_{2m})\begin{pmatrix} \lambda^{-1} & 0 \\ 0 & 1 \end{pmatrix}.\]
\noindent Donc, $M_{n}(\lambda a_{1},\lambda^{-1} a_{2},\ldots,\lambda a_{2m-1},\lambda^{-1} a_{2m})=\epsilon \id$.
\end{proof}

\noindent On a de plus quelques égalités concernant les matrices $M_{n}(a_{1},\ldots,a_{n})$.

\begin{lemma}[\cite{CH}, proposition 4.1 et lemme 4.2]
\label{form}
Soient $A$ un anneau commutatif unitaire et\\$(a,b,u,v,x,y) \in A^{6}$ avec $a$ et $uv-1$ inversibles dans $A$. On a :

\smallskip
i) $M_{3}(x,0,y)=-M_{1}(x+y)$;

\smallskip
ii) $M_{3}(x,1,y)=M_{2}(x-1,y-1)$;

\smallskip
iii) $M_{3}(x,-1,y)=-M_{2}(x+1,y+1)$;

\smallskip
iv) $M_{4}(x,u,v,y)=M_{3}(x+(1-v)(uv-1)^{-1},uv-1,y+(1-u)(uv-1)^{-1})$;

\smallskip
v) $M_{5}(x,a,a^{-1},b,y)=M_{3}((a^{2}x-2a+b)a^{-2},-a,(ay-1)a^{-1})$.
\end{lemma}

\begin{proof}

Ces formules découlent d'un calcul direct. Par ailleurs, i) et iv) sont donnés dans le lemme 4.2 de \cite{CH} et ii) et iii) sont donnés dans la proposition 4.1 de \cite{CH}.
\end{proof}

\subsubsection{Cas $n$ impair}

Soient $p$ un nombre premier impair et $q$ une puissance de $p$. Le théorème \ref{36} nous donne $u_{n, q}^{-}=\left[\frac{n-1}{2}\right]_{q^{2}}$. Pour démontrer le point i) du théorème \ref{25}, il reste donc à montrer que $u_{n, q}^{-}=u_{n, q}^{+}$.

\begin{proposition}
\label{44}
Soient $n \in \bN^{*}$, $n$ impair, $A$ un anneau commutatif unitaire, $\Omega_{n}(A)$ l'ensemble des solutions de taille $n$ de l'équation $M_{n}(a_{1},\ldots,a_{n})= \id$ sur $A$ et $\Xi_{n}(A)$ l'ensemble des solutions de taille $n$ de $M_{n}(a_{1},\ldots,a_{n})= -\id$ sur $A$. L'application 
\[\begin{array}{ccccc}  
\varphi_{n} & : & \Omega_{n}(A) & \longrightarrow & \Xi_{n}(A) \\
 & & (a_{1},\ldots,a_{n}) & \longmapsto & (-a_{1},\ldots,-a_{n})  \\
\end{array}\] est une bijection.
\end{proposition}
\begin{proof}
Avec
\[ K:=\begin{pmatrix} -1 & 0 \\ 0 & 1 \end{pmatrix},
\quad
K \begin{pmatrix} a & -1 \\ 1 & 0 \end{pmatrix} K
= -\begin{pmatrix} -a & -1 \\ 1 & 0 \end{pmatrix},
\]
\noindent on obtient pour $(a_1,\ldots,a_n)\in \Omega_{n}(A)$
\begin{eqnarray*}
\id = K^2 = K M_n(a_1,\ldots,a_n) K &=& K M_1(a_n) K \ldots K M_{1}(a_1) K \\
&=& (-1)^{n} M_n(-a_1,\ldots,-a_n) \\
&=& - M_n(-a_1,\ldots,-a_{n-1},-a_n)
\end{eqnarray*}
puisque $n$ est impair.
\end{proof}

\noindent Cette proposition donne $u_{n,q}^{-}=u_{n,q}^{+}$ pour $n$ impair et $A=\mathbb{F}_q$, donc la partie i) du théorème \ref{25}.

\section{Démonstration du théorème \ref{25}}
\label{preuve25}

\subsection{Preuve directe}

Soient $n$ un entier naturel, $p$ un nombre premier impair et $q$ une puissance de $p$. Si $n$ est impair alors les formules contenues dans le théorème \ref{25} sont vraies (voir section \ref{chap41}). On suppose maintenant $n=2m$.
\\
\\Notons $\Omega_{n}(q)$ l'ensemble des solutions de taille $n$ de $M_{n}(a_{1},\ldots,a_{n})= \id$ sur $\mathbb{F}_{q}$ et $\Xi_{n}(q)$ l'ensemble des solutions de taille $n$ de $M_{n}(a_{1},\ldots,a_{n})= -\id$ sur $\mathbb{F}_{q}$. $u_{2m,q}^{-}=\left|\Xi_{n}(q)\right|$ est donné par le théorème \ref{36}.
\\
\\ Pour terminer la preuve, il nous reste à calculer $u_{2m,q}^{+}$. On a :
\begin{eqnarray*}
\Omega_{n}(q) &=& \{(a_{1},\ldots,a_{n}) \in \Omega_{n}(q),~a_{2}=0\} \sqcup \{(a_{1},\ldots,a_{n}) \in \Omega_{n}(q),~a_{2} \neq 0\} \\
                         &=& \bigsqcup_{a \in \mathbb{F}_{q}} \underbrace{\{(a_{1},\ldots,a_{n}) \in \Omega_{n}(q),~a_{1}=a~{\rm et}~a_{2}=0\}}_{\Psi_{a}} \bigsqcup_{b \in \mathbb{F}_{q}^{*}} \underbrace{\{(a_{1},\ldots,a_{n}) \in \Omega_{n}(q),~a_{2}=b\}}_{\Theta_{b}} .\\
\end{eqnarray*}
												
\noindent Considérons, pour commencer, les deux applications suivantes définies pour $a \in \mathbb{F}_{q}$ : \[\begin{array}{ccccc} 
\varphi_{n,a} : & \Psi_{a} & \longrightarrow & \Xi_{n-2}(q) \\
  & (a,0,a_{3},\ldots,a_{2m}) & \longmapsto & (a+a_{3},\ldots,a_{2m})  \\
\end{array}~{\rm et}~\begin{array}{ccccc} 
\phi_{n,a}  : & \Xi_{n-2}(q) & \longrightarrow & \Psi_{a}~~~~~~~~~~~~~~. \\
 & (a_{1},\ldots,a_{2m-2}) & \longmapsto & (a,0,a_{1}-a,a_{2},\ldots,a_{2m})  \\
\end{array}\]

\noindent Par le lemme \ref{form} i), ces deux applications sont bien définies. De plus, ce sont des bijections réciproques. Donc, $\left|\Psi_{a}\right|=\left|\Xi_{n-2}(q)\right|$ et $\left|\{(a_{1},\ldots,a_{n}) \in \Omega_{n}(q),~a_{2}=0\}\right|=qu_{2m-2,q}^{-}$.
\\
\\ \noindent Considérons, maintenant, les deux applications suivantes définies pour $b \in \mathbb{F}_{q}^{*}$: \[\begin{array}{ccccc} 
\vartheta_{n,b} : & \Theta_{b} & \longrightarrow & \Omega_{n-1}(q) \\
  & (a_{1},b,a_{3},\ldots,a_{2m}) & \longmapsto & (a_{1}b-1,a_{3}b-1,a_{4}b^{-1},a_{5}b,\ldots,a_{2m}b^{-1})  \\
\end{array}\]
\[{\rm et}~~\begin{array}{ccccc} 
\theta_{n,b}  : & \Omega_{n-1}(q) & \longrightarrow & \Theta_{b} \\
 & (a_{1},\ldots,a_{2m-1}) & \longmapsto & ((a_{1}+1)b^{-1},b,(a_{2}+1)b^{-1},a_{3}b,\ldots,a_{2m-1}b)  \\
\end{array}.\]

\noindent Par le lemme \ref{48} et le lemme \ref{form} ii), ces deux applications sont bien définies. De plus, ce sont des bijections réciproques. Donc, $\left|\Theta_{b}\right|=\left|\Omega_{n-1}(q)\right|$ et $\left|\{(a_{1},\ldots,a_{n}) \in \Omega_{n}(q),~a_{2} \neq 0\}\right|=(q-1)u_{2m-1,q}^{+}$.
\\
\\On aboutit ainsi à la formule de récurrence suivante : $u_{2m,q}^{+}=(q-1)u_{2m-1,q}^{+}+q u_{2m-2,q}^{-}$.
\\
\\En particulier, puisque $u_{2m-1,q}^{+}=u_{2m-1,q}^{-}$ (voir proposition \ref{44}), on a $u_{2m,q}^{+}=(q-1)u_{2m-1,q}^{-}+q u_{2m-2,q}^{-}$. Or, $u_{2m-1,q}^{-}$ et $u_{2m-2,q}^{-}$ sont connus. Ainsi, si $m$ est pair, on a :
\[u_{n,q}^{+}=(q-1)\frac{q^{2m-2}-1}{q^{2}-1}+q(q-1)\binom{m-1}{2}_{q}+q^{m-1}=(q-1)\binom{m}{2}_{q}+q^{m-1}.\]

\noindent De même, si $m$ est impair, on a :
\[u_{n,q}^{+}=(q-1)\frac{q^{2m-2}-1}{q^{2}-1}+q(q-1)\binom{m-1}{2}_{q}=(q-1)\binom{m}{2}_{q}.\]

\noindent Ainsi, les formules données dans le théorème \ref{25} sont vraies, ce qui conclut la preuve de ce dernier. 
\qed

\begin{remark}

{\rm En utilisant les mêmes arguments, on peut démontrer la formule suivante :} 
\[u_{2m,q}^{-}=(q-1)u_{2m-1,q}^{-}+q u_{2m-2,q}^{+}.\] 

\end{remark}

\subsection{La formule générale de récurrence}
\label{chap44}

Nous allons maintenant montrer le théorème \ref{thm412}, c'est-à-dire établir une formule générale de récurrence reliant les solutions de $M_{n}(a_{1},\ldots,a_{n})=B$ et celles de $M_{n}(a_{1},\ldots,a_{n})=-B$ pour différentes valeurs de $n$, avec $B$ une matrice quelconque de $\SL_{2}(\mathbb{F}_{q})$.

\begin{proof}[Démonstration du théorème \ref{thm412}]

Soit $n>4$ et $a \in \mathbb{F}_{q}$. On définit les ensembles suivants :
\begin{itemize}
\item $\Omega_{n}^{B}(q)=\{(a_{1},\ldots,a_{n}) \in \mathbb{F}_{q}^{n},~M_{n}(a_{1},\ldots,a_{n})=B\}$;
\item $\Lambda_{n}^{B}(q)=\{(a_{1},\ldots,a_{n}) \in \Omega_{n}^{B}(q),~a_{2}=0\}$;
\item $\Psi_{n,a}^{B}(q)=\{(a_{1},\ldots,a_{n}) \in \Lambda_{n}^{B}(q),~a_{1}=a\}$;
\item $\Delta_{n}^{B}(q)=\{(a_{1},\ldots,a_{n}) \in \Omega_{n}^{B}(q),~a_{2} \neq 0\}$.
\end{itemize}

\noindent Pour commencer, on va démontrer l' égalité suivante : $\left|\Lambda_{n}^{B}(q)\right|=qu_{n-2,q}^{-B}$. \noindent Considérons, pour cela, les deux applications suivantes définies pour $a \in \mathbb{F}_{q}$ : \[\begin{array}{ccccc} 
\varphi_{n,a} : & \Psi_{n,a}^{B}(q) & \longrightarrow & \Omega_{n-2}^{-B}(q) \\
  & (a,0,a_{3},\ldots,a_{n}) & \longmapsto & (a+a_{3},\ldots,a_{n})  \\
\end{array}~{\rm et}~\begin{array}{ccccc} 
\phi_{n,a}  : & \Omega_{n-2}^{-B}(q) & \longrightarrow & \Psi_{n,a}^{B}(q)~~~~~~~~~~~~~~. \\
 & (a_{1},\ldots,a_{n-2}) & \longmapsto & (a,0,a_{1}-a,a_{2},\ldots,a_{n-2})  \\
\end{array}\]

\noindent Par le lemme \ref{form} i), ces deux applications sont bien définies. De plus, ce sont des bijections réciproques. Donc, $\left|\Psi_{n,a}^{B}(q)\right|=\left|\Omega_{n-2}^{-B}(q)\right|=u_{n-2,q}^{-B}$ et $\left|\Lambda_{n}^{B}(q)\right|=qu_{n-2,q}^{-B}$.
\\
\\De plus, on dispose de la relation suivante : $\left|\Delta_{n}^{B}(q)\right|=\left|\Omega_{n}^{B}(q)\right|-\left|\Lambda_{n}^{B}(q)\right|$. On a :

\[\Omega_{n}^{B}(q)=\{(a_{1},\ldots,a_{n}) \in \Omega_{n}^{B}(q),~a_{2}a_{3} \neq 1\} \sqcup \{(a_{1},\ldots,a_{n}) \in \Omega_{n}^{B}(q),~a_{2}a_{3}=1\}.\]

\noindent i) $\{(a_{1},\ldots,a_{n}) \in \Omega_{n}^{B}(q),~a_{2}a_{3} \neq 1\}=\Lambda_{n}^{B}(q) \sqcup \left(\bigsqcup_{b \in \mathbb{F}_{q}^{*}} \underbrace{\{(a_{1},\ldots,a_{n}) \in \Omega_{n}^{B}(q),~a_{2}=b,~a_{2}a_{3} \neq 1\}}_{Y_{b}}\right)$.
\\
\\ On pose maintenant pour $b \in \mathbb{F}_{q}^{*}$ :
\[\begin{array}{ccccc} 
f_{n,b} : & Y_{b} & \longrightarrow & \Delta_{n-1}^{B}(q) \\
  & (a_{1},b,a_{3},\ldots,a_{n}) & \longmapsto & (a_{1}+(1-a_{3})(ba_{3}-1)^{-1},ba_{3}-1,a_{4}+(1-b)(ba_{3}-1)^{-1},a_{5},\ldots,a_{n})  \\
\end{array}\]
\[{\rm et}~~\begin{array}{ccccc} 
g_{n,b} : & \Delta_{n-1}^{B}(q) & \longrightarrow & Y_{b} \\
 & (a_{1},\ldots,a_{n-1}) & \longmapsto & (a_{1}+(b^{-1}(a_{2}+1)-1)a_{2}^{-1},b,b^{-1}(a_{2}+1),a_{3}+(b-1)a_{2}^{-1},a_{4},\ldots,a_{n-1})  \\
\end{array}.\]

\noindent Par le lemme \ref{form} iv), $f_{n,b}$ et $g_{n,b}$ sont bien définies. Ce sont des bijections réciproques. Donc, on a l'égalité $\left|Y_{b}\right|=\left|\Delta_{n-1}^{B}(q)\right|$. Ainsi, 
\begin{eqnarray*}
\left|\{(a_{1},\ldots,a_{n}) \in \Omega_{n}^{B}(q),~a_{2}a_{3} \neq 1\}\right| &=& \left|\Lambda_{n}^{B}(q)\right|+(q-1)\left|\Delta_{n-1}^{B}(q)\right| \\
                                                                               &=& qu_{n-2,q}^{-B}+(q-1)(\left|\Omega_{n-1}^{B}(q)\right|-\left|\Lambda_{n-1}^{B}(q)\right|) \\
																																							&=&	qu_{n-2,q}^{-B}+(q-1)(u_{n-1,q}^{B}-qu_{n-3,q}^{-B}).	\\
\end{eqnarray*}

\noindent ii) $\{(a_{1},\ldots,a_{n}) \in \Omega_{n}^{B}(q),~a_{2}a_{3}=1\}=\bigsqcup_{x \in \mathbb{F}_{q}} \underbrace{\{(x,a_{2},a_{2}^{-1},a_{4},\ldots,a_{n}) \in \Omega_{n}^{B}(q)\}}_{Z_{x}}$.
\\
\\ On pose maintenant pour $x \in \mathbb{F}_{q}$ :
\[\begin{array}{ccccc} 
h_{n,x} : & Z_{x} & \longrightarrow & \Delta_{n-2}^{B}(q) \\
  & (x,a_{2},a_{2}^{-1},a_{4},a_{5},\ldots,a_{n}) & \longmapsto & ((a_{2}^{2}x-2a_{2}+a_{4})a_{2}^{-2},-a_{2},(a_{2}a_{5}-1)a_{2}^{-1},a_{6}\ldots,a_{n})  \\
\end{array}\] 
\[{\rm et}~~\begin{array}{ccccc} 
k_{n,x} : & \Delta_{n-2}^{B}(q) & \longrightarrow & Z_{x} \\
 & (a_{1},\ldots,a_{n-2}) & \longmapsto & (x,-a_{2},-a_{2}^{-1},a_{2}^{2}(a_{1}-x)-2a_{2},(a_{3}a_{2}-1)a_{2}^{-1},a_{4},\ldots,a_{n-2})  \\
\end{array}.\]

\noindent Par le lemme \ref{form} v), $h_{n,x}$ et $k_{n,x}$ sont bien définies. Ce sont des bijections réciproques. Donc, on a l'égalité $\left|Z_{x}\right|=\left|\Delta_{n-2}^{B}(q)\right|$. Donc, 
\begin{eqnarray*}
\left|\{(a_{1},\ldots,a_{n}) \in \Omega_{n}^{B}(q),~a_{2}a_{3}=1\}\right| &=& q\left|\Delta_{n-2}^{B}(q)\right| \\
                                                                               &=& q(\left|\Omega_{n-2}^{B}(q)\right|-\left|\Lambda_{n-2}^{B}(q)\right|) \\
																																							&=&	q(u_{n-2,q}^{B}-qu_{n-4,q}^{-B}).	\\
\end{eqnarray*}

\noindent iii) En regroupant, les résultats de i) et ii), on a :

\[u_{n,q}^{B}=\left|\Omega_{n}^{B}(q)\right|=qu_{n-2,q}^{-B}+(q-1)(u_{n-1,q}^{B}-qu_{n-3,q}^{-B})+q(u_{n-2,q}^{B}-qu_{n-4,q}^{-B}).\] 
\end{proof}

\noindent Muni de cette relation de récurrence et des petites valeurs de $u_{n,q}^{+}$ et $u_{n,q}^{-}$ évoquées dans la section \ref{quid}, on peut retrouver simplement les théorèmes \ref{25} et \ref{36}. En effet, il suffit de vérifier que les formules données dans les deux énoncés vérifient bien les relations de récurrence que nous venons d'établir. De plus, contrairement à la preuve directe, cette démonstration utilise seulement les formules données dans le lemme \ref{form} et les solutions de \eqref{p} pour les petites valeurs de $n$ (en particulier on n'utilise pas le théorème de S.\ Morier-Genoud).

\begin{remark}
{\rm Les matrices $\pm \id$ sont les seules matrices scalaires que l'on peut considérer (puisque $\lambda \id \in \SL_{2}(\mathbb{F}_{q})$ si et seulement si $\lambda=\pm 1$). Toutefois, on peut chercher, grâce au théorème \ref{thm412}, des formules de comptage pour d'autres matrices.
}
\end{remark}

\section{Nombre de $\lambda$-quiddités sur $\mathbb{Z}/N\mathbb{Z}$ pour certains entiers $N$ non premiers}
\label{preuve26}

L'objectif de cette section est d'obtenir une expression du nombre de $\lambda$-quiddités sur les anneaux $\mathbb{Z}/N\mathbb{Z}$ pour certains entiers $N$ non premiers, en particulier pour $N=4$.

\subsection{Démonstration du théorème \ref{26}}

\begin{proof}

Soient $B \in \SL_{2}(\bZ/4\bZ)$ et $n \in \mathbb{N}^{*}$. On note :
\begin{itemize}
\item $\tilde{\Omega}_{n}^{B}(4)=\{(\overline{a_{1}},\ldots,\overline{a_{n}}) \in (\bZ/4\bZ)^{n},~M_{n}(\overline{a_{1}},\ldots,\overline{a_{n}})=B\}$;
\item $w_{n,4}^{B}$ le cardinal de $\tilde{\Omega}_{n}^{B}(4)$.
\end{itemize}
Soit $n \geq 4$. On va utiliser les solutions de $(E_{\bZ/4\bZ})$ de taille inférieure à 4 pour obtenir une relation de récurrence permettant de calculer les cardinaux souhaités. Pour cela, on va regarder les différentes valeurs possibles de la deuxième composante d'une solution.

\begin{itemize}

\item $(\overline{a_{1}},\overline{1},\overline{a_{3}},\ldots,\overline{a_{n}}) \longmapsto (\overline{a_{1}-1},\overline{a_{3}-1},\overline{a_{4}},\ldots,\overline{a_{n}})$ établit une bijection entre l'ensemble des éléments de $\tilde{\Omega}_{n}^{B}(4)$ dont la deuxième composante est $\overline{1}$ et $\tilde{\Omega}_{n-1}^{B}(4)$. Ainsi, il y a $w_{n-1,4}^{B}$ éléments de $\tilde{\Omega}_{n}^{B}(4)$  possédant un $\overline{1}$ en deuxième position.
\\
\item $(\overline{a_{1}},\overline{-1},\overline{a_{3}},\ldots,\overline{a_{n}}) \longmapsto (\overline{a_{1}+1},\overline{a_{3}+1},\overline{a_{4}},\ldots,\overline{a_{n}})$ établit une bijection entre l'ensemble des éléments de $\tilde{\Omega}_{n}^{B}(4)$ dont la deuxième composante est $\overline{-1}$ et $\tilde{\Omega}_{n-1}^{-B}(4)$. Ainsi, il y a $w_{n-1,4}^{-B}$ éléments de $\tilde{\Omega}_{n}^{B}(4)$  possédant un $\overline{-1}$ en deuxième position.
\\
\item En raisonnant de la même façon que dans la preuve du théorème \ref{thm412}, on obtient qu'il y a $4w_{n-2,4}^{-B}$ éléments dans $\tilde{\Omega}_{n}^{B}(4)$ possédant un $\overline{0}$ en deuxième position. 
\\
\item Pour le cas où la deuxième composante vaut $\overline{2}$, on regarde la troisième composante.

\begin{itemize}[label=$\circ$]

\item $(\overline{a_{1}},\overline{2},\overline{1},\overline{a_{4}},\ldots,\overline{a_{n}}) \longmapsto (\overline{a_{1}-1},\overline{a_{4}-2},\overline{a_{5}},\ldots,\overline{a_{n}})$ établit une bijection entre l'ensemble des éléments de $\tilde{\Omega}_{n}^{B}(4)$ dont la deuxième composante vaut $\overline{2}$ et la troisième $\overline{1}$ et $\tilde{\Omega}_{n-2}^{B}(4)$. Cela donne $w_{n-2,4}^{B}$ éléments de $\tilde{\Omega}_{n}^{B}(4)$ .
\\

\item $(\overline{a_{1}},\overline{2},\overline{-1},\overline{a_{4}},\ldots,\overline{a_{n}}) \longmapsto (\overline{a_{1}+1},\overline{a_{4}-2},\overline{a_{5}},\ldots,\overline{a_{n}})$ établit une bijection entre l'ensemble des éléments de $\tilde{\Omega}_{n}^{B}(4)$ dont la deuxième composante vaut $\overline{2}$ et la troisième $\overline{-1}$ et $\tilde{\Omega}_{n-2}^{B}(4)$. Cela donne $w_{n-2,4}^{B}$ éléments de $\tilde{\Omega}_{n}^{B}(4)$ .
\\

\item $(\overline{a_{1}},\overline{2},\overline{0},\overline{a_{4}},\ldots,\overline{a_{n}}) \longmapsto (\overline{a_{1}},\overline{a_{4}-2},\overline{a_{5}},\ldots,\overline{a_{n}})$ établit une bijection entre l'ensemble des éléments de $\tilde{\Omega}_{n}^{B}(4)$ dont la deuxième composante vaut $\overline{2}$ et la troisième $\overline{0}$ et $\tilde{\Omega}_{n-2}^{-B}(4)$. Cela donne $w_{n-2,4}^{-B}$ éléments de $\tilde{\Omega}_{n}^{B}(4)$ .
\\

\item $(\overline{a_{1}},\overline{2},\overline{2},\overline{a_{4}},\ldots,\overline{a_{n}}) \longmapsto (\overline{a_{1}-2},\overline{a_{4}-2},\overline{a_{5}},\ldots,\overline{a_{n}})$  établit une bijection entre l'ensemble des éléments de $\tilde{\Omega}_{n}^{B}(4)$ dont la deuxième composante vaut $\overline{2}$ et la troisième $\overline{2}$ et $\tilde{\Omega}_{n-2}^{-B}(4)$. Cela donne $w_{n-2,4}^{-B}$ éléments de $\tilde{\Omega}_{n}^{B}(4)$ .
\end{itemize}

\end{itemize}

\noindent Ainsi, on a la formule de récurrence suivante : \[w_{n,4}^{B}=w_{n-1,4}^{B}+w_{n-1,4}^{-B}+6w_{n-2,4}^{-B}+2w_{n-2,4}^{B}.\]
 
\noindent De même, on a $w_{n,4}^{-B}=w_{n-1,4}^{B}+w_{n-1,4}^{-B}+6w_{n-2,4}^{B}+2w_{n-2,4}^{-B}$. 
\\
\\On pose $C=\begin{pmatrix}
               1 & 1 & 2 &6 \\
							 1 & 1 &6 &2 \\
							 1 & 0 &0 &0 \\
							 0 &1 &0 &0 \\
							\end{pmatrix}$, $P=\begin{pmatrix}
               1 & 1 & 1 & 1 \\
							 1 & 1 & -1 & -1 \\
							 1/4 & -1/2 & i/2 & -i/2 \\
							 1/4 & -1/2 & -i/2 & i/2 \\
							\end{pmatrix}$ et $D=\begin{pmatrix}
               4 & 0 & 0 & 0 \\
							 0 & -2 & 0 & 0 \\
							 0 & 0 & -2i & 0 \\
							 0 & 0 & 0 & 2i \\
							\end{pmatrix}$. $P$ est inversible et $P^{-1}=\begin{pmatrix}
               1/3 & 1/3 & 2/3 & 2/3 \\
							 1/6 & 1/6  & -2/3 & -2/3 \\
							 1/4 & -1/4 & -i/2 & i/2 \\
							 1/4 & -1/4 & i/2& -i/2 \\ 
							\end{pmatrix}$. $C$ est diagonalisable dans $\bC$ et $C=PDP^{-1}$. On a \[\begin{pmatrix}
               w_{n,4}^{B} \\
							 w_{n,4}^{-B}  \\
							 w_{n-1,4}^{B}  \\
							 w_{n-1,4}^{-B}  \\
							\end{pmatrix}= PD^{n-3}P^{-1}\begin{pmatrix}
               w_{3,4}^{B} \\
							 w_{3,4}^{-B}  \\
							 w_{2,4}^{B}  \\
							 w_{2,4}^{-B}  \\
							\end{pmatrix}.\]
\\
\\Puisque $w_{3,4}^{\id}=w_{3,4}^{-\id}=1$, $w_{2,4}^{\id}=0$ et $w_{2,4}^{-\id}=1$, on obtient :
\[\left\{
\begin{array}{rcr}
w_{n,4}^{+} &=& \frac{4^{n-2}}{3}-\frac{(-2)^{n-3}}{3}-i^{n-2} 2^{n-4}-(-i)^{n-2}2^{n-4}; \\
w_{n,4}^{-} &=& \frac{4^{n-2}}{3}-\frac{(-2)^{n-3}}{3}+i^{n-2} 2^{n-4}+(-i)^{n-2}2^{n-4}.
\end{array}
\right.\] 

\noindent On considère $2^{n-4}(i^{n-2}+(-1)^{n-2}i^{n-2})$. On distingue trois cas : 
\begin{itemize}
\item si $n$ est impair alors $n-2$ est impair et $2^{n-4}(i^{n-2}+(-1)^{n-2}i^{n-2})=2^{n-4}(i^{n-2}-i^{n-2})=0$;
\item si $n=2m$ avec $m$ pair. $n-2$ est pair non divisible par 4 donc on a l'égalité : \[2^{n-4}(i^{n-2}+(-1)^{n-2}i^{n-2})=2^{n-4}(-1+(-1)))=-2^{n-3};\]
\item si $n=2m$ avec $m$ impair. $n-2$ est divisible par 4 donc on dispose de l'égalité : \[2^{n-4}(i^{n-2}+(-1)^{n-2}i^{n-2})=2^{n-4}(1+1)=2^{n-3}.\]

\end{itemize}

\noindent En reportant dans les formules de $w_{n,4}^{+}$ et $w_{n,4}^{-}$, on obtient : 

\begin{itemize}
\item si $n$ est impair alors on a l'égalité suivante : 
\[w_{n, 4}^{+}=w_{n, 4}^{-}=\frac{4^{n-2}-2^{n-3}}{3};\]

\item si $n=2m$ avec $m$ pair, on a :
\[w_{n,4}^{+}=\frac{4^{n-2}+2^{n-1}}{3}~~~{\rm et}~~~w_{n,4}^{-}=\frac{4^{n-2}- 2^{n-2}}{3};\]

\item si $n=2m$ avec $m$ impair, on a :
\[w_{n,4}^{+}=\frac{4^{n-2}- 2^{n-2}}{3}~~~{\rm et}~~~w_{n,4}^{-}=\frac{4^{n-2}+2^{n-1}}{3}.\]

\end{itemize}
\end{proof}

\begin{remark}
{\rm On peut utiliser les éléments précédents pour trouver des formules de comptage concernant les générateurs du groupe modulaire.
\\i) Si $B=S=\begin{pmatrix}
              \overline{0} & \overline{-1} \\
							 \overline{1} & \overline{0}
							\end{pmatrix}$ alors $w_{n,4}^{-S}=w_{n-1,4}^{S}+w_{n-1,4}^{-S}+6w_{n-2,4}^{S}+2w_{n-2,4}^{-S}$. Comme $w_{2,4}^{S}=w_{2,4}^{-S}=0$ et $w_{3,4}^{S}=0$ $w_{3,4}^{-S}=4$, on a les formules générales suivantes :
\[w_{n,4}^{S}=\frac{4^{n-2}-(-2)^{n-2}}{3}+2^{n-3}i^{n-3}((-1)^{n-2}-1)~~~~w_{n,4}^{-S}=\frac{4^{n-2}-(-2)^{n-2}}{3}+2^{n-3}i^{n-3}(1+(-1)^{n-3}).\]

\noindent ii) Si $B=T=\begin{pmatrix}
              \overline{1} & \overline{1} \\
							 \overline{0} & \overline{1}
							\end{pmatrix}$ alors $w_{n,4}^{-T}=w_{n-1,4}^{T}+w_{n-1,4}^{-T}+6w_{n-2,4}^{T}+2w_{n-2,4}^{-T}$.Comme $w_{2,4}^{T}=0$, $w_{2,4}^{-T}=1$ et $w_{3,4}^{T}=w_{3,4}^{-T}=1$, on a	les formules générales suivantes :
\[w_{n,4}^{T}=\frac{4^{n-2}-(-2)^{n-3}}{3}+2^{n-4}i^{n-2}((-1)^{n-3}-1)~~~~w_{n,4}^{-T}=\frac{4^{n-2}-(-2)^{n-3}}{3}+2^{n-4}i^{n-2}(1+(-1)^{n-2}).\]				
}

\noindent Cela donne notamment les valeurs numériques ci-dessous :

\begin{center}
\begin{tabular}{|c|c|c|c|c|c|c|c|c|c|}
\hline
  $n$     & 2 & 3 & 4 & 5 & 6 & 7 & 8 & 9 & 10  \rule[-7pt]{0pt}{18pt} \\
  \hline
  $w_{n,4}^{S}$ & 0 & 0 & 4 & 32 & 80 & 320 & 1344  & 5632 & 21~760   \rule[-7pt]{0pt}{18pt} \\
	\hline
  $w_{n,4}^{-S}$ & 0 & 4 & 4 & 16 & 80 & 384 & 1344 & 5376 & 21~760   \rule[-7pt]{0pt}{18pt} \\
	\hline
  $w_{n,4}^{T}$ & 0 & 1 & 8 & 20 & 80 & 336 & 1408 & 5440 & 21~760    \rule[-7pt]{0pt}{18pt} \\
	\hline
  $w_{n,4}^{-T}$  & 1 & 1 & 4 & 20 & 96 & 336 & 1344 & 5440 & 22~016 \rule[-7pt]{0pt}{18pt} \\
	\hline
	
\end{tabular}
\end{center}

\end{remark}

\subsection{Application au calcul du nombre de solutions pour d'autres valeurs de $N$}

Soit $N$ un entier naturel supérieur à 2. On note :
\begin{itemize}
\item $\tilde{\Omega}_{n}(N)=\{(\overline{a_{1}},\ldots,\overline{a_{n}}) \in (\bZ/N\bZ)^{n},~M_{n}(\overline{a_{1}},\ldots,\overline{a_{n}})=\id\}$;
\item $\tilde{\Xi}_{n}(N)=\{(\overline{a_{1}},\ldots,\overline{a_{n}}) \in (\bZ/N\bZ)^{n},~M_{n}(\overline{a_{1}},\ldots,\overline{a_{n}})=-\id\}$.
\end{itemize}

Comme les coefficients de $M_{n}(\overline{a_{1}},\ldots,\overline{a_{n}})$ sont des polynômes en $\overline{a_{1}},\ldots,\overline{a_{n}}$, on a avec le lemme chinois :

\begin{proposition}
\label{414}

Soit $N=p_{1}^{\alpha_{1}}\ldots p_{r}^{\alpha_{r}}$ avec les $p_{i}$ des nombres premiers deux à deux distincts et les $\alpha_{i}$ des entiers naturels non nuls. Soit $n$ un entier naturel supérieur à 2. Les applications 
\[\begin{array}{cccc} 
\tau : & \tilde{\Omega}_{n}(N) & \longrightarrow & \tilde{\Omega}_{n}(p_{1}^{\alpha_{1}}) \times \ldots \times \tilde{\Omega}_{n}(p_{r}^{\alpha_{r}}) \\
 & (a_{1}+N\bZ,\ldots,a_{n}+N\bZ) & \longmapsto & ((a_{1}+p_{1}^{\alpha_{1}}\bZ,\ldots,a_{n}+p_{1}^{\alpha_{1}}\bZ),\ldots,(a_{1}+p_{r}^{\alpha_{r}}\bZ,\ldots,a_{n}+p_{r}^{\alpha_{r}}\bZ))  \\
\end{array}\]
et
\[\begin{array}{cccc} 
\delta : & \tilde{\Xi}_{n}(N) & \longrightarrow & \tilde{\Xi}_{n}(p_{1}^{\alpha_{1}}) \times \ldots \times \tilde{\Xi}_{n}(p_{r}^{\alpha_{r}}) \\
 & (a_{1}+N\bZ,\ldots,a_{n}+N\bZ) & \longmapsto & ((a_{1}+p_{1}^{\alpha_{1}}\bZ,\ldots,a_{n}+p_{1}^{\alpha_{1}}\bZ),\ldots,(a_{1}+p_{r}^{\alpha_{r}}\bZ,\ldots,a_{n}+p_{r}^{\alpha_{r}}\bZ))  \\
\end{array}\]

\noindent sont des bijections.
\\
\\En particulier, $\left|\tilde{\Omega}_{n}(N)\right|= \prod_{i=1}^{r} \left| \tilde{\Omega}_{n}(p_{i}^{\alpha_{i}}) \right|$ et $\left|\tilde{\Xi}_{n}(N)\right|= \prod_{i=1}^{r} \left| \tilde{\Xi}_{n}(p_{i}^{\alpha_{i}}) \right|$.

\end{proposition}

Dans le cas où la décomposition de $N$ en nombres premiers ne contient pas de facteur carré, ou contient seulement 4 comme facteur carré, on peut utiliser les théorèmes \ref{25}, \ref{36} et \ref{26} pour obtenir le nombre de $\lambda$-quiddités de \eqref{p} sur $\bZ/N\bZ$. En effet, en combinant ce résultat avec les théorèmes \ref{25}, \ref{36} et \ref{26}, on obtient le corollaire \ref{261}.

\subsection{Applications numériques}

On donne ci-dessous quelques valeurs numériques pour $w_{n,N}^{+}$ et $w_{n,N}^{-}$. On commence par donner des valeurs de $w_{n,N}^{-}$ :

\hfill\break

\begin{center}
\begin{tabular}{|c|c|c|c|c|c|c|c|c|c|}
\hline
  \multicolumn{1}{|c|}{\backslashbox{$n$}{\vrule width 0pt height 1.25em$N$}}     & 2 & 3 & 4 & 5 & 6 & 7 & 10 & 11 & 12  \rule[-7pt]{0pt}{18pt} \\
  \hline
  4 & 3 & 2 & 4 & 4 & 6 & 6 & 12 & 10 & 8   \rule[-7pt]{0pt}{18pt} \\
	\hline
  5 & 5 & 10 & 20 & 26 & 50 & 50 & 130 & 122 & 200   \rule[-7pt]{0pt}{18pt} \\
	\hline
  6  & 11 & 35 & 96 & 149 & 385 & 391 & 1639 & 1451 & 3360    \rule[-7pt]{0pt}{18pt} \\
	\hline
  7  & 2 & 91 & 336 & 651 & 1911 & 2451 & 13~671 & 14~763 & 30~576 \rule[-7pt]{0pt}{18pt} \\
  \hline
	8  & 43 & 260 & 1344 & 3224 & 11~180 & 17~100 & 138~632 & 162~260 & 349~440  \rule[-7pt]{0pt}{18pt} \\
	\hline
	
\end{tabular}
\end{center}

\hfill\break

\noindent On donne maintenant des valeurs de $w_{n,N}^{+}$ :

\hfill\break

\begin{center}
\begin{tabular}{|c|c|c|c|c|c|c|c|c|c|}
\hline
  \multicolumn{1}{|c|}{\backslashbox{$n$}{\vrule width 0pt height 1.25em$N$}}     & 2 & 3 & 4 & 5 & 6 & 7 & 10 & 11 & 12  \rule[-7pt]{0pt}{18pt} \\
  \hline
  4 & 3 & 5 & 8 & 9 & 15 & 13 & 27 & 21 & 40   \rule[-7pt]{0pt}{18pt} \\
	\hline
  5 & 5 & 10 & 20 & 26 & 50 & 50 & 130 & 122 & 200   \rule[-7pt]{0pt}{18pt} \\
	\hline
  6  & 11 & 26 & 80 & 124 & 286 & 342 & 1364 & 130 & 2080   \rule[-7pt]{0pt}{18pt} \\
	\hline
  7  & 21 & 91 & 336 & 651 & 1911 & 2451 & 13~671 & 14~763 & 30~576 \rule[-7pt]{0pt}{18pt} \\
  \hline
	8  & 43 & 287 & 1408 & 3349 & 12~341 & 17~443 & 148~307 & 163~591 & 404~096     \rule[-7pt]{0pt}{18pt} \\
	\hline
	
\end{tabular}
\end{center}

\hfill\break

\noindent Notons que les valeurs de $w_{n,4}^{+}$ et $w_{n,4}^{-}$ sont différentes de celles de $u_{n,4}^{+}$ et $u_{n,4}^{-}$ (voir \cite{Mo2} section 3.8).

\section{Quelques éléments sur le nombre de classe d'équivalence des solutions irréductibles}
\label{preuve27}

L'objectif de cette section est de donner quelques résultats concernant la suite $(v_{N})$ qui compte le nombre de classes d'équivalence de solutions irréductibles de \eqref{p} sur $\bZ/N\bZ$.

\subsection{Démonstration du théorème \ref{27}}

\noindent On s'intéresse ici à des propriétés asymptotiques de cette suite. Avant de démontrer le théorème évoqué dans l'introduction, on donne le résultat préliminaire suivant : 

\begin{lemma}[\cite{M1}, propositions 3.4 et 3.8]
\label{52}

Une solution de \eqref{p} de taille 4 sur $\bZ/N\bZ$ est réductible si et seulement si elle contient $\overline{1}$ ou $\overline{-1}$.

\end{lemma}

\begin{proof}

Soit $(\overline{a_{1}},\overline{a_{2}},\overline{a_{3}},\overline{a_{4}})$ une solution de \eqref{p} sur $\bZ/N\bZ$. 

Si $(\overline{a_{1}},\overline{a_{2}},\overline{a_{3}},\overline{a_{4}})$ est réductible alors $(\overline{a_{1}},\overline{a_{2}},\overline{a_{3}},\overline{a_{4}})$ est équivalent à la somme d'un $m$-uplet solution de \eqref{p} avec un $l$-uplet solution de \eqref{p} avec $m, l \geq 3$. On a $m+l-2=4$ donc $m+l=6$ et comme $m, l \geq 3$ on a nécessairement $m=l=3$. Comme les solutions de \eqref{p} de taille 3 contiennent $\overline{1}$ ou $\overline{-1}$, une solution réductible de taille 4 contient $\overline{1}$ ou $\overline{-1}$.

Si $(\overline{a_{1}},\overline{a_{2}},\overline{a_{3}},\overline{a_{4}})$ est une solution de \eqref{p} contenant $\overline{1}$ ou $\overline{-1}$. Il existe $\epsilon$ dans $\{-1, 1\}$ il existe $i$ dans $[\![1;4]\!]$ tels que $\overline{a_{i}}=\overline{\epsilon}$. Quitte à effectuer une permutation circulaire, on peut supposer $i=4$. On a :\[(\overline{a_{1}},\overline{a_{2}},\overline{a_{3}},\overline{a_{4}})=(\overline{a_{1}-\epsilon},\overline{a_{2}},\overline{a_{3}-\epsilon}) \oplus (\overline{\epsilon},\overline{\epsilon},\overline{\epsilon}).\]
Comme $(\overline{\epsilon},\overline{\epsilon},\overline{\epsilon})$ est solution de \eqref{p}, $(\overline{a_{1}},\overline{a_{2}},\overline{a_{3}},\overline{a_{4}})$ est une solution réductible.
\end{proof}

\noindent On peut maintenant effectuer la preuve annoncée.

\begin{proof}[Démonstration du théorème \ref{27}]

i) Posons $N=2^{2m}$ avec $m \geq 2$ et $n=2m$. On va compter le nombre de solutions irréductibles de \eqref{p} que l'on connaît afin de minorer celui-ci.
\\
\\On s'intéresse aux solutions du type $(\overline{a},\overline{b},\overline{-a},\overline{-b})$ avec $\overline{ab}=\overline{0}$. 
\begin{itemize}
\item On a déjà les solutions $(\overline{a},\overline{0},\overline{-a},\overline{0})$ avec $\overline{a} \neq \pm \overline{1}$ qui donnent $\frac{N}{2}$ classes d'équivalence. En effet, $(\overline{a},\overline{0},\overline{-a},\overline{0}) \sim (\overline{b},\overline{0},\overline{-b},\overline{0})$ si et seulement si $\overline{a}=\pm \overline{b}$. Donc, pour $0 \leq a \leq \frac{N}{2}$, $a \neq 1$, on a des classes de solutions irréductibles deux à deux distinctes.
\\
\item On a aussi les solutions $(\overline{2^{m+k}a},\overline{2^{m-k}b},\overline{-2^{m+k}a},\overline{-2^{m-k}b})$ avec $0 \leq a \leq 2^{m-k}$ et $0 \leq b \leq 2^{m+k}$, $a, b$ impairs et $1 \leq k \leq m-1$. Pour chaque $1 \leq k \leq m-1$, il y a $2^{m-k-1}2^{m+k-1}=2^{n-2}$ solutions de ce type. Par ailleurs, ces solutions sont irréductibles car elles ne contiennent pas $\pm \overline{1}$. Notons également que ces solutions ne sont pas équivalentes aux solutions évoquées précédemment car elles ne contiennent pas $\overline{0}$. Deux solutions de ce type équivalentes utilisent nécessairement la même valeur de $k$. Soit $1 \leq k \leq m-1$. Soient $0 \leq a \leq 2^{m-k}$ et $0 \leq b \leq 2^{m+k}$ avec $a, b$ impair, on a au plus quatre solutions de ce type appartenant à la même classe : 
\begin{itemize}
\item $(\overline{2^{m+k}a},\overline{2^{m-k}b},\overline{-2^{m+k}a},\overline{-2^{m-k}b})$;
\item $(\overline{2^{m+k}(2^{m-k}-a)},\overline{2^{m-k}b},\overline{-2^{m+k}(2^{m-k}-a)},\overline{-2^{m-k}b})$; 
\item $(\overline{2^{m+k}a},\overline{2^{m-k}(2^{m+k}-b)},\overline{-2^{m+k}a},\overline{-2^{m-k}(2^{m+k}-b)})$;\item $(\overline{2^{m+k}(2^{m-k}-a)},\overline{2^{m-k}(2^{m+k}-b)},\overline{-2^{m+k}(2^{m-k}-a)},\overline{-2^{m-k}(2^{m+k}-b)})$. 
\end{itemize}

\end{itemize}

\noindent Ainsi, on a au moins $\sum_{k=1}^{m-1} 2^{n-4}=(m-1)2^{n-4}=\frac{{\rm ln}(N)}{32{\rm ln}(2)}N-\frac{N}{16}$ classes de solutions irréductibles deux à deux distinctes de cette forme. Donc, \[v_{N} \geq \frac{{\rm ln}(N)}{32{\rm ln}(2)}N-\frac{N}{16}+\frac{N}{2}=\frac{{\rm ln}(N)}{32{\rm ln}(2)}N+\frac{7N}{16}.\] En particulier, 
\[\frac{v_{N}}{N\sqrt{{\rm ln}(N)}} \geq \frac{\sqrt{{\rm ln}(N)}}{32{\rm ln}(2)}+\frac{7}{16\sqrt{{\rm ln}(N)}}.\]

\noindent Comme il y a une infinité d'entiers de la forme $2^{2m}$, on peut extraire une sous-suite de $\left(\frac{v_{N}}{N\sqrt{{\rm ln}(N)}}\right)$ non bornée (puisque le terme de droite de l'inégalité tend vers $+\infty$). Ainsi, $\left(\frac{v_{N}}{N\sqrt{{\rm ln}(N)}}\right)$ n'est pas bornée.
\\
\\ ii) Supposons par l'absurde qu'il existe $l \in \mathbb{R}^{+}$ tel que $(v_{N}-lN)$ est bornée. Il existe $K$ tel que, $\forall N \geq 2$, $v_{N}-lN \leq K$. En particulier, $\frac{v_{N}}{N\sqrt{{\rm ln}(N)}} \leq \frac{K}{N\sqrt{{\rm ln}(N)}}+\frac{l}{\sqrt{{\rm ln}(N)}} \underset{N\to+\infty}{\longrightarrow} 0$. Donc, $\left(\frac{v_{N}}{N\sqrt{{\rm ln}(N)}}\right)$ est bornée, ce qui est absurde.
\end{proof}

\begin{remark}
On connaît d'autres solutions irréductibles de $(E_{2^{2m}})$. On dispose notamment pour ces équations d'une classification de toutes les solutions monomiales minimales irréductibles (voir \cite{M5} Théorème 2.6).
\end{remark}

\subsection{Calcul numérique des valeurs de $v_{N}$ pour les petites valeurs de $N$}

On connaît déjà l'ensemble des solutions irréductibles de \eqref{p} lorsque $A=\bZ/N\bZ$ pour $2 \leq N \leq 6$ (voir \cite{M1} Théorème 2.5), ce qui nous permet d'avoir les valeurs de $v_{N}$ pour ces mêmes entiers. À l'aide d'un programme informatique, fourni dans l'annexe \ref{A}, on étend cette connaissance aux valeurs de $v_{N}$ pour $N \leq 16$.
La longueur des plus longues solutions irréductibles est notée $\ell_N$.

\begin{center}
\begin{tabular}{|c|c|c|c|c|c|c|c|c|c|c|c|c|c|c|c|}
\hline
$N$     & 2 & 3 & 4 & 5 & 6 & 7 & 8 & 9 & 10 & 11 & 12 & 13 & 14 & 15 & 16  \rule[-7pt]{0pt}{18pt} \\
\hline
$v_{N}$ & 2 & 3 & 6 & 9 & 10 & 42 & 48 & 229 & 203 & 25~686 & 1161 & 2~913~226 & 90~748 & 14~346~911 & 8~259~494  \rule[-7pt]{0pt}{18pt} \\
\hline
$\ell_N$ & 4 & 4 & 4 & 6 & 6 & 9 & 8 & 12 & 12 & 19 & 15 & 25 & 20 & 26 & 24
\rule[-7pt]{0pt}{18pt} \\
  \hline
\end{tabular}
\end{center}

\hfill\break

\noindent On constate, en particulier, que $v_{10} < v_{9}$. Donc, la suite $(v_{N})_{N \geq 2}$ n'est pas croissante.
\\
\\ \noindent {\bf Remerciements}.
Les auteurs remercient Sophie Morier-Genoud pour sa disponibilité et son aide précieuse. Ils remercient également le rapporteur pour ses commentaires et ses nombreuses propositions de modifications qui ont été utilisés pour améliorer ce texte d'une façon notable.

\appendix

\section{Démonstration du théorème \ref{25} à l'aide d'une bijection}
\label{ann_preuve25}
\label{chap42}

Nous allons redémontrer le théorème \ref{25} en donnant une bijection entre l'ensemble des solutions et un ensemble dont on pourra calculer le cardinal. Cette bijection étant pratiquement la même que celle donnée dans \cite{Mo2} (seules quelques signes sont à changer), on se contente de rappeler les grandes lignes de sa construction. Cette dernière sera effectuée sur un corps commutatif $\mathbb{K}$ quelconque. En revanche, les éléments de comptage seront effectués sur $\mathbb{F}_{q}$. Ici, $p$ est un nombre premier impair, $q$ une puissance de $p$ et $n$ un entier naturel non nul pair de la forme $n=2m$. 

\subsection{Les ensembles $\mathcal{C}_{n}$}

On commence par introduire un certain nombre de notations utiles pour la suite :
\begin{itemize}
\item $\Omega_{n}(\mathbb{K})$ est l'ensemble des solutions de taille $n$ de $M_{n}(a_{1},\ldots,a_{n})= \id$ sur $\mathbb{K}$;
\item $\Xi_{n}(\mathbb{K})$ est l'ensemble des solutions de taille $n$ de $M_{n}(a_{1},\ldots,a_{n})= -\id$ sur $\mathbb{K}$;
\item $\mathcal{D(\mathbb{K})}$ désigne l'ensemble des droites vectorielles de $\mathbb{K}^{2}$; 
\item $r \in \bM$, $\mathcal{C}_{r}(\mathbb{K})=\{(d_{1},\ldots,d_{r}) \in \mathcal{D(\mathbb{K})}^{r}, d_{r} \neq d_{1}~{\rm et}~d_{i} \neq d_{i+1}~{\rm pour}~i \in [\![1;r-1]\!]\}$.
\item $\Omega_{n}(q)=\Omega_{n}(\mathbb{F}_{q})$, $\Xi_{n}(q)=\Xi_{n}(\mathbb{F}_{q})$ et $\mathcal{C}_{r}(q)=\mathcal{C}_{r}(\mathbb{F}_{q})$.
\\
\end{itemize}

\noindent On va maintenant définir deux sous-ensembles de $\mathcal{C}_{2m}(\mathbb{K})$ qui nous seront très utiles pour la suite. Si $(d_{1},\ldots,d_{2m}) \in \mathcal{C}_{n}(\mathbb{K})$, on pose $d_{j+n}=d_{j}$. Soit $(d_{1},\ldots,d_{2m}) \in \mathcal{C}_{n}(\mathbb{K})$. Pour tout $i$ dans $[\![1;n]\!]$, il existe un $v_{i}$ dans $\mathbb{K}^{2}$ tel que $d_{i}={\rm vect}(v_{i})$. Notons $V=(v_{1},\ldots,v_{n})$. Soit $\mathcal{B}$ une base de $\mathbb{K}^{2}$. Comme $d_{i} \neq d_{i+1}$ pour $i \in [\![1;n]\!]$, ${\rm det}_{\mathcal{B}}(v_{i},v_{i+1}) \neq 0$. On peut donc considérer le quotient suivant :

\[f(d_{1},\ldots,d_{2m}):=\frac{{\rm det}_{\mathcal{B}}(v_{1},v_{2}){\rm det}_{\mathcal{B}}(v_{3},v_{4})\ldots{\rm det}_{\mathcal{B}}(v_{2m-1},v_{2m})}{{\rm det}_{\mathcal{B}}(v_{2},v_{3}){\rm det}_{\mathcal{B}}(v_{4},v_{5})\ldots{\rm det}_{\mathcal{B}}(v_{2m},v_{1})}.\]
\[~\]
\noindent Notons que $f(d_{1},\ldots,d_{2m})$ ne dépend ni du choix de la base ni du choix des vecteurs directeurs. On peut maintenant introduire les deux ensembles ci-dessous (voir \cite{Mo2} section 2.2) :
\[\mathcal{C}_{2m}(\mathbb{K})^{+}=\{(d_{1},\ldots,d_{2m}) \in \mathcal{C}_{2m}(\mathbb{K}), f(d_{1},\ldots,d_{2m})=-1\},\]
\[\mathcal{C}_{2m}(\mathbb{K})^{-}=\{(d_{1},\ldots,d_{2m}) \in \mathcal{C}_{2m}(\mathbb{K}), f(d_{1},\ldots,d_{2m})=1\}.\]

\noindent On peut définir de façon naturelle une action de $\mathbb{K}^{*}$ sur $\Omega_{n}(\mathbb{K})$ et une action de ${\rm PGL}(\mathbb{K}^{2})$ sur $\mathcal{C}_{2m}(\mathbb{K})^{-}$ (les détails de ces dernières sont données dans les sections qui suivent). Le but est de définir une bijection $\sigma$ entre les orbites de ces actions, c'est-à-dire de montrer que :
\[\Omega_{2m}(\mathbb{K})/\mathbb{K}^{*} \cong \mathcal{C}_{2m}(\mathbb{K})^{-}/{\rm PGL}(\mathbb{K}^{2}).\]

\noindent On va maintenant donner quelques formules de comptage sur $\mathbb{F}_{q}$. On dispose notamment de la formule de dénombrement suivante :

\begin{lemma}[\cite{Mo2}, lemme 3.1]
\label{45}

Soit $r \geq 2$, on a $\left|\mathcal{C}_{r}(q)\right|=q^{r}+(-1)^{r}q$.

\end{lemma}

\noindent De plus, les cardinaux de $\mathcal{C}_{r}(q)^{+}$ et de $\mathcal{C}_{r}(q)^{-}$ sont liés par la formule de dénombrement ci-dessous :

\begin{lemma}[\cite{Mo2}, lemme 3.2]
\label{46}

Soit $r \geq 4$ pair, on a 
\begin{itemize}
\item $\left|\mathcal{C}_{r}(q)^{+}\right|=\left|\mathcal{C}_{r-1}(q)\right|+q\left|\mathcal{C}_{r-2}(q)^{-}\right|$;
\item $\left|\mathcal{C}_{r}(q)^{-}\right|=\left|\mathcal{C}_{r-1}(q)\right|+q\left|\mathcal{C}_{r-2}(q)^{+}\right|$.
\end{itemize}

\end{lemma}

\noindent Notons, en particulier, que pour $r \geq 4$, $\left|\mathcal{C}_{r}(q)^{+}\right| \neq 0$ et $\left|\mathcal{C}_{r}(q)^{-}\right| \neq 0$. Par ailleurs, tous les couples de droites distinctes appartiennent à $\mathcal{C}_{2}(q)^{+}$. Donc, $\left|\mathcal{C}_{2}(q)^{+}\right| \neq 0$. En revanche, $\mathcal{C}_{2}(q)^{-}=\emptyset$.
\\
\\ \noindent Soient $u \in {\rm GL}(\mathbb{K}^{2})$, $m \geq 2$ et $(d_{1},\ldots,d_{2m}) \in \mathcal{C}_{2m}(\mathbb{K})^{-}$ avec $d_{i}={\rm vect}(v_{i})$. Comme $u$ est inversible, on a, pour tout $i$ dans $[\![1;n]\!]$, $u(d_{i}) \neq u(d_{i+1})$. De plus, en utilisant la formule ${\rm det}_{\mathcal{B}}(u(v_{i}),u(v_{i+1}))={\rm det}(u){\rm det}_{\mathcal{B}}(v_{i},v_{i+1})$, on a $f(u(d_{1}),\ldots,u(d_{2m}))=1$. Ainsi, on peut définir une action de ${\rm PGL}(\mathbb{K}^{2})$ sur $\mathcal{C}_{2m}(\mathbb{K})^{-}$ de la façon suivante : \[\begin{array}{ccccc} 
\alpha & : & {\rm PGL}(\mathbb{K}^{2}) \times \mathcal{C}_{2m}(\mathbb{K})^{-} & \longrightarrow & \mathcal{C}_{2m}(\mathbb{K})^{-} \\
 & & (\{\lambda u, \lambda \in \mathbb{K}^{*}\}, (d_{1},\ldots,d_{2m})) & \longmapsto & (u(d_{1}),\ldots,u(d_{2m}))  \\
\end{array}.\]

\noindent On note $\widehat{\mathcal{M}}_{2m}(\mathbb{K})^{-}$ les orbites de $\mathcal{C}_{2m}(\mathbb{K})^{-}$ sous cette action. Si $m=1$, on pose $\widehat{\mathcal{M}}_{2m}(\mathbb{K})^{-}=\emptyset$. On procède de manière analogue pour $\mathcal{C}_{2m}(\mathbb{K})^{+}$ en désignant par $\widehat{\mathcal{M}}_{2m}(\mathbb{K})^{+}$ les orbites. On dispose du résultat suivant :

\begin{lemma}[\cite{Mo2}, Théorème 4]
\label{47}

Soit $p$ un nombre premier impair et $q$ une puissance de $p$. On a pour $m \geq 1$ :
\begin{itemize}
\item si $m$ pair, $\left|\widehat{\mathcal{M}}_{2m}(q)^{+}\right|=\binom{m}{2}_{q}$;
\item si $m$ impair, $\left|\widehat{\mathcal{M}}_{2m}(q)^{+}\right|=\binom{m}{2}_{q}+[m-1]_{q}+1$.
\end{itemize}

\end{lemma}

Nous rappelons qu'une application projective linéaire dans l'espace projectif $\mathbb{P}^r$ est déterminée par l'image de $r+1$ points.
À partir des résultats précédents, on peut calculer $\left|\widehat{\mathcal{M}}_{2m}(q)^{-}\right|$.

\begin{lemma}
\label{comptage}

Soient $p$ un nombre premier impair, $q$ une puissance de $p$ et $n=2m$. On a pour $m \geq 1$ :

\begin{itemize}
\item si $m$ est impair $\left|\widehat{\mathcal{M}}_{2m}(q)^{-}\right|=\binom{m}{2}_{q}$;
\item si $m$ est pair, $\left|\widehat{\mathcal{M}}_{2m}(q)^{-}\right|=\binom{m}{2}_{q}+q[m-2]_{q}+2$.
\end{itemize}

\end{lemma}

\begin{proof}

\uwave{Supposons $m$ impair.} Si $m=1$ alors $\left|\widehat{\mathcal{M}}_{2m}(q)^{-}\right|=0$. Supposons $m$ supérieur à 3. Comme $m-1$ est pair et non nul, $\mathcal{C}_{2m-2}(q)^{+}$ ne contient pas d'élément de la forme $(d_{1},d_{2},\ldots,d_{1},d_{2})$. Donc, les éléments de $\mathcal{C}_{2m-2}(q)^{+}$ contiennent toujours au moins trois droites distinctes (par définition de $\mathcal{C}_{2m-2}(q)$). Ainsi, les orbites sous l'action de ${\rm PGL}(\mathbb{F}_{q}^{2})$ ont $\left|{\rm PGL}(\mathbb{F}_{q}^{2})\right|$ éléments. Donc, \[\left|\widehat{\mathcal{M}}_{2m-2}(q)^{+}\right|=\frac{\left|\mathcal{C}_{2m-2}(q)^{+}\right|}{\left|{\rm PGL}(\mathbb{F}_{q}^{2})\right|}=\frac{\left|\mathcal{C}_{2m-2}(q)^{+}\right|}{q(q^{2}-1)}.\]

\noindent Par le lemme \ref{47}, $\left|\mathcal{C}_{2m-2}(q)^{+}\right|=q(q^{2}-1)\binom{m-1}{2}_{q}$ ($m-1 \geq 1$).
\\
\\Par les lemmes \ref{45} et \ref{46}, $\left|\mathcal{C}_{2m}(q)^{-}\right|=q^{n-1}-q+q^{2}(q^{2}-1)\binom{m-1}{2}_{q}$ ($2m \geq 4$).
\\
\\Comme $m$ est impair, $\mathcal{C}_{2m}(q)^{-}$ ne contient pas d'élément de la forme $(d_{1},d_{2},\ldots,d_{1},d_{2})$. Donc, les éléments de $\mathcal{C}_{2m}(q)^{-}$ contiennent toujours au moins trois droites distinctes et les orbites sous l'action de ${\rm PGL}(\mathbb{F}_{q}^{2})$ ont $\left|{\rm PGL}(\mathbb{F}_{q}^{2})\right|$ éléments. Ainsi, \[\left|\widehat{\mathcal{M}}_{2m}(q)^{-}\right|=\frac{\left|\mathcal{C}_{2m}(q)^{-}\right|}{\left|{\rm PGL}(\mathbb{F}_{q}^{2})\right|}=\frac{-1+q^{n-2}}{q^{2}-1}+q\binom{m-1}{2}_{q}=\binom{m}{2}_{q}.\]

\noindent \uwave{Supposons $m$ pair.} Comme $m-1$ est impair, $\mathcal{C}_{2m-2}(q)^{+}$ contient les éléments de la forme $(d_{1},d_{2},\ldots,d_{1},d_{2})$ qui sont tous dans la même orbite. Les autres éléments contiennent toujours au moins trois droites distinctes et leur orbite sous l'action de ${\rm PGL}(\mathbb{F}_{q}^{2})$ ont $\left|{\rm PGL}(\mathbb{F}_{q}^{2})\right|$ éléments. Il y a $q+1$ droites vectorielles dans $\mathbb{F}_{q}^{2}$, donc il y a $q(q+1)$ éléments de la forme $(d_{1},d_{2},\ldots,d_{1},d_{2})$. Donc, \[\left|\widehat{\mathcal{M}}_{2m-2}(q)^{+}\right|=\frac{\left|\mathcal{C}_{2m-2}(q)^{+}\right|-q(q+1)}{\left|{\rm PGL}(\mathbb{F}_{q}^{2})\right|}+1=\frac{\left|\mathcal{C}_{2m-2}(q)^{+}\right|-q(q+1)}{q(q^{2}-1)}+1.\]

\noindent Par le lemme \ref{47}, $\left|\mathcal{C}_{2m-2}(q)^{+}\right|=(\binom{m-1}{2}_{q}+[m-2]_{q})q(q^{2}-1)+q(q+1)$ ($m-1 \geq 1$).
\\
\\Par les lemmes \ref{45} et \ref{46}, $\left|\mathcal{C}_{2m}(q)^{-}\right|=q^{n-1}-q+(\binom{m-1}{2}_{q}+[m-2]_{q})q^{2}(q^{2}-1)+q^{2}(q+1)$ ($2m \geq 4$).
\\
\\Comme $m$ est pair, $\mathcal{C}_{2m}(q)^{-}$ contient les éléments de la forme $(d_{1},d_{2},\ldots,d_{1},d_{2})$ qui sont tous dans la même orbite. Les autres éléments de $\mathcal{C}_{2m}(q)^{-}$ contiennent toujours au moins trois droites distinctes et les orbites sous l'action de ${\rm PGL}(\mathbb{F}_{q}^{2})$ ont $\left|{\rm PGL}(\mathbb{F}_{q}^{2})\right|$ éléments. Ainsi, \[\left|\widehat{\mathcal{M}}_{2m}(q)^{-}\right|=\frac{\left|\mathcal{C}_{2m}(q)^{-}\right|-q(q+1)}{\left|{\rm PGL}(\mathbb{F}_{q}^{2})\right|}+1=\frac{q^{n-2}-1}{q^{2}-1}+q\left(\binom{m -1}{2}_{q}+[m-2]_{q}\right)+2=\binom{m}{2}_{q}+q[m-2]_{q}+2.\]
\end{proof}

\subsection{Les ensembles $\widehat{\Omega}_{n}(\mathbb{K})$}
\label{action}

À la lueur de la proposition \ref{48}, on peut définir une action de $\mathbb{K}^{*}$ sur $\Omega_{n}(\mathbb{K})$ en posant :
\[\begin{array}{ccccc} 
\beta & : & \mathbb{K}^{*} \times \Omega_{n}(\mathbb{K}) & \longrightarrow & \Omega_{n}(\mathbb{K}) \\
 & & (\lambda, (a_{1},\ldots,a_{n})) & \longmapsto & (\lambda a_{1},\lambda^{-1} a_{2},\ldots,\lambda a_{2m-1},\lambda^{-1} a_{2m})  \\
\end{array}.\]

\noindent On note $\widehat{\Omega}_{n}(\mathbb{K})$ les orbites de cette action. 

\begin{lemma}
\label{481}

Soient $n \geq 4$, $n=2m$, $p$ un nombre premier impair et $q$ une puissance de $p$.
 
\begin{itemize}
\item Si $m$ est pair alors $\left|\Omega_{n}(q)\right|=1+(q-1)(\left|\widehat{\Omega}_{n}(q)\right|-1)$;
\item Si $m$ est impair alors $\left|\Omega_{n}(q)\right|=(q-1)\left|\widehat{\Omega}_{n}(q)\right|$.
\end{itemize}

\end{lemma}

\begin{proof}

Si $m$ est pair alors l'orbite de $(0,\ldots,0)$ ne contient qu'un seul élément tandis que les autres en contiennent $q-1$. En appliquant l'équation des classes, on obtient la formule souhaitée. Si $m$ est impair, $(0,\ldots,0)$ n'est pas solution et chaque orbite contient $q-1$ éléments. En appliquant l'équation des classes, on obtient la formule souhaitée.
\end{proof}

\subsection{Établissement d'une bijection}

Avant de redémontrer le théorème \ref{25}, on a besoin de plusieurs résultats intermédiaires. Leur preuve étant très similaire aux démonstrations des lemmes de la section 2 de \cite{Mo2}, on se contente ici de donner les énoncés. Dans cette sous-section, on considère $m \geq 2$.

\begin{lemma}
\label{49}

Soit $(d_{1},\ldots,d_{2m}) \in \mathcal{C}_{2m}(\mathbb{K})$. Il existe un choix de vecteurs $(v_{1},\ldots,v_{2m})$ relevant $(d_{1},\ldots,d_{2m})$ tel que pour toute base $\mathcal{B}$ \[{\rm det}_{\mathcal{B}}(v_{1},v_{2})={\rm det}_{\mathcal{B}}(v_{2},v_{3})=\ldots={\rm det}_{\mathcal{B}}(v_{2m-1},v_{2m})={\rm det}_{\mathcal{B}}(v_{2m},v_{1})\]
\noindent si et seulement si $(d_{1},\ldots,d_{2m}) \in \mathcal{C}_{2m}(\mathbb{K})^{-}$.

\end{lemma}

\begin{lemma}
\label{410}

Soient $(d_{1},\ldots,d_{2m}) \in \mathcal{C}_{2m}(\mathbb{K})^{-}$ et $(v_{1},\ldots,v_{2m})$ un choix de vecteurs  directeurs vérifiant ${\rm det}(v_{1},v_{2})={\rm det}(v_{2},v_{3})=\ldots={\rm det}(v_{2m-1},v_{2m})={\rm det}(v_{2m},v_{1})$. Il existe un unique $n$-uplet $(a_{1},\ldots,a_{n})$ d'éléments de $\mathbb{K}$ vérifiant $v_{i}=a_{i}v_{i-1}-v_{i-2}$ (pour $i \in [\![1;n]\!]$) avec $v_{0}=v_{n}$ et $v_{-1}=v_{n-1}$.
\\
\\De plus, pour tout $\lambda \in \mathbb{K}^{*}$, $(\lambda a_{1},\lambda^{-1} a_{2},\ldots,\lambda a_{2m-1},\lambda^{-1} a_{2m}) \in \Omega_{n}(\mathbb{K})$.

\end{lemma}

\noindent Soient $D \in \widehat{\mathcal{M}}_{2m}(\mathbb{K})^{-}$ et $(d_{1},\ldots,d_{2m}) \in D$. Soit $(v_{1},\ldots,v_{2m})$ un choix de vecteurs directeurs pour $(d_{1},\ldots,d_{2m})$ vérifiant les conditions du lemme \ref{49}. À partir de ce choix, on définit, grâce au lemme \ref{410}, des éléments $(a_{1},\ldots,a_{n})$ vérifiant $v_{i}=a_{i}v_{i-1}-v_{i-2}$. Notons $B$ l'orbite de ces éléments sous l'action de $\mathbb{K}^{*}$. On montre que $B$ ne dépend que de $D$. On peut ainsi définir une application
\[ \sigma : \widehat{\mathcal{M}}_{2m}(\mathbb{K})^{-} \longrightarrow \widehat{\Omega}_{n}(\mathbb{K}), \quad D\longmapsto B. \]

\begin{lemma}
\label{411}
$\sigma$ est une bijection.
\end{lemma}

\subsection{Calcul de $u_{n,q}^{+}$}

On peut maintenant terminer la preuve du théorème \ref{25}. Par les résultats de la section \ref{chap41}, le théorème est vrai pour $p$ premier impair et $n$ impair. Soient $n \geq 4$, $n=2m$, $p$ un nombre premier impair et $q$ une puissance de $p$.
\\
\\i) Supposons $m$ pair. Par le lemme \ref{481}, $u_{n,q}^{+}=\left|\Omega_{n}(q)\right|=1+(q-1)(\left|\widehat{\Omega}_{n}(q)\right|-1)$. Or, par le lemme \ref{411}, $\left|\widehat{\Omega}_{n}(q)\right|=\left|\widehat{\mathcal{M}}_{2m}(q)^{-}\right|$. De plus, par le lemme \ref{comptage}, $\left|\widehat{\mathcal{M}}_{2m}(q)^{-}\right|=\binom{m}{2}_{q}+q[m-2]_{q}+2$. Donc, 
\[u_{n,q}^{+}=1+(q-1)\left(\binom{m}{2}_{q}+q[m-2]_{q}+2-1\right)=1+(q-1)\binom{m}{2}_{q}+q^{m-1}-q+q-1=(q-1)\binom{m}{2}_{q}+q^{m-1}.\]

\noindent ii) Supposons $m$ impair. Par le lemme \ref{481}, $u_{n,q}^{+}=\left|\Omega_{n}(q)\right|=(q-1)(\left|\widehat{\Omega}_{n}(q)\right|)$. Or, par le lemme \ref{411}, on a $\left|\widehat{\Omega}_{n}(q)\right|=\left|\widehat{\mathcal{M}}_{2m}(q)^{-}\right|$. De plus, par le lemme \ref{comptage}, $\left|\widehat{\mathcal{M}}_{2m}(q)^{-}\right|=\binom{m}{2}_{q}$. Donc, $u_{n,q}^{+}=(q-1)\binom{m}{2}_{q}$.
\qed

\begin{remark}
{\rm On peut aussi interpréter les ensembles $\widehat{\mathcal{M}}_{2m}(q)^{-}$ et $\widehat{\mathcal{M}}_{2m}(q)^{+}$ comme des variétés, voir par exemple \cite{Mo2}.
}
\end{remark}

\section{Algorithme pour calculer le nombre de classes de solutions irréductibles de \eqref{p} sur $\bZ/N\bZ$}
\label{A}

L'algorithme \ref{algo_irr} est un prototype simplifié du programme utilisé pour le calcul numérique des valeurs de $v_N$ pour les petites valeurs de $N$. Pour obtenir toutes les quiddités irréductibles pour un anneau $A$, on appelle $\text{\tt Irréductibles}((),(0),(0,1))$.
Le principe de base est très simple: on ajoute des nombres à $c$ et on  calcule les nouvelles colonnes de la frise en même temps.
Pour obtenir les valeurs associées à $N=13,15,16$, il faut optimiser la technique.

\algo{Irréductibles}{$c=(c_1,\ldots,c_n)$, $u=(u_1,\ldots,u_{n+1})$, $v=(v_1,\ldots,v_{n+2})$}
{Calcule toutes les quiddités irréductibles qui commencent par $c$}
{le début $c$, les deux dernières colonnes $u,v$ de la frise correspondante}
{liste de quiddités irréductibles}
{
\item $R:=\emptyset$
\item for $x$ in $A$:
\item \quad $d:=(c_1,\ldots,c_n,x)$
\item \quad if $(x,c_n,...,c_1) \ge d$:
\item \quad \quad $w:=(0,1,x) \cat (x v_i - u_{i-1} \mid i \in \{3,\ldots,n+2\})$
\item \quad \quad if $\pm 1 \in \{w_3,\ldots,w_{n+3}\}$:
\item \quad \quad \quad if $\pm 1 \notin \{w_3,\ldots,w_{n+2}\}$:
\item \quad \quad \quad \quad $\varepsilon := w_{n+3}$
\item \quad \quad \quad \quad $y:=\varepsilon v_{n+2}$
\item \quad \quad \quad \quad $w':= (0,1,y) \cat (y w_i - v_{i-1} \mid i \in \{3,\ldots,n+3\})$
\item \quad \quad \quad \quad if $\pm 1 \notin \{w'_3,\ldots,w'_{n+2}\}:$
\item \quad \quad \quad \quad \quad $y:=\varepsilon v_{n+2}$
\item \quad \quad \quad \quad \quad $z:=\varepsilon w_{n+2}$
\item \quad \quad \quad \quad \quad $w'':=(0,1,z) \cat (z w'_i - w_{i-1} \mid i \in \{3,\ldots,n+4\})$
\item \quad \quad \quad \quad \quad if $\pm 1 \notin \{w''_3,\ldots,w''_{n+2}\}$ and $w''_{n+5}=-\varepsilon$:
\item \quad \quad \quad \quad \quad \quad $c' :=$ représentant de $(c_1,\ldots,c_n,x,y,z)$ sous l'action du groupe diédral
\item \quad \quad \quad \quad \quad \quad $R := R \cup \{c\}$
\item \quad \quad \quad \quad \quad end if
\item \quad \quad \quad \quad end if
\item \quad \quad \quad end if
\item \quad \quad else
\item \quad \quad \quad $R:=R \cup \text{Irréductibles}(d,v,w)$
\item \quad \quad end if
\item \quad end if
\item end for
\item return $R$
}{algo_irr}

\end{document}